\documentclass[12pt, reqno]{amsart}
\usepackage{geometry}
\usepackage{ucs}
\usepackage[utf8x]{inputenc}
\usepackage[english]{babel}
\usepackage{tikz}
\usetikzlibrary{matrix}
\usepackage{bbding}

\renewcommand{\Re}{\mrm{Re}}
\usepackage{lineno}

\AtBeginDocument{
  \addtocontents{toc}{\small}
}

\usepackage{ifpdf}
\usepackage{amsmath}
\usepackage{amsthm, amssymb,bm}
\usepackage{enumerate}
\usepackage{bbm}
\usepackage{hyperref}
\usepackage{float}
\usepackage{color}
\usepackage{dsfont}
\usepackage[all,cmtip]{xy}
\hypersetup{
  colorlinks=true,
  citecolor=black,
  linkcolor=black,
  urlcolor=black,
  filecolor=red}
\usepackage{upgreek}

\newcommand{\Bmu}{\bm{\upmu}}
 
\numberwithin{equation}{section}
\renewcommand{\theequation}{\thesection.\alph{equation}}

\newcounter{everything}
\numberwithin{everything}{section}

\theoremstyle{plain}

\newtheorem{theo}[everything]{\sc Theorem}
\newtheorem{mainthm}{\sc Theorem}
\newtheorem{propo}[everything]{\sc Proposition}
\newtheorem{corol}[everything]{\sc Corollary}
\newtheorem{lem}[everything]{\sc Lemma}
\theoremstyle{remark}

\newtheorem{rem}[everything]{\sc Remark}

\theoremstyle{definition}
\newtheorem*{defi}{\sc Definition}

\renewcommand{\hat}{\widehat}
\renewcommand{\tilde}{\widetilde}
\renewcommand{\bar}{\overline}

\newcommand{\mfr}[1]{\mathfrak{#1}}
\newcommand{\mbf}[1]{\mathbf{#1}}
\newcommand{\msf}[1]{\mathsf{#1}}
\newcommand{\mrm}[1]{\mathrm{#1}}
\newcommand{\mcal}[1]{\mathcal{#1}}

\newcommand{\gen}[1]{\langle{#1}\rangle}
\newcommand{\set}[1]{\left\{{#1}\right\}}

\newcommand{\abs}[1]{\left|#1\right|}

\DeclareMathOperator{\val}{\emph{val}}

\newcommand{\dual}{{\bm{\vee}}}

\newcommand{\dpar}[1]{\!\left(\!\left(#1\right)\!\right)}
\newcommand{\dbN}{\mathbb N}
\newcommand{\dbZ}{\mathbb Z}
\newcommand{\dbF}{\mathbb F}
\newcommand{\dbQ}{\mathbb Q}
\newcommand{\dbR}{\mathbb R}
\newcommand{\dbC}{\mathbb C}

\newcommand{\matr}{\mathbf{M}}

\newcommand{\GL}{\mrm{GL}}
\newcommand{\SL}{\mrm{SL}}
\newcommand{\GU}{\mrm{GU}}

\newcommand{\GG}{\mathsf{G}}
\newcommand{\g}{\msf{g}}
\newcommand{\pp}{\mfr{p}}
\newcommand{\oo}{\mfr{o}}
\newcommand{\PP}{\mfr{P}}
\newcommand{\OO}{\mfr{O}}

\newcommand{\Gal}{\mathsf{Gal}}
\newcommand{\TT}{\mcal{T}}
\newcommand{\CC}{\mbf{C}}

\newcommand{\St}{\mbf{St}}
\newcommand{\bO}{\mbf{O}}

\newcommand{\mat}[1]{\left(\begin{matrix}#1\end{matrix}\right)}

\newcommand{\ceil}[1]{\lceil#1\rceil}

\renewcommand{\ker}{\mathrm{Ker}}

\newcommand{\irr}{\mathrm{Irr}}
\renewcommand{\hom}{\mathrm{Hom}}

\newcommand{\stab}{\textstyle\mathop{Stab}}

\newcommand{\res}{\mathrm{Res}}
\newcommand{\ind}{\mathrm{Ind}}

\newcommand{\Tr}{\mathrm{Tr}}
\newcommand{\Nr}{\mathrm{Nr}}
\newcommand{\Trd}{\mathrm{Trd}}
\newcommand{\Nrd}{\mathrm{Nrd}}

\newcommand{\NNrd}{\msf{Nrd}}
\newcommand{\TTrd}{\msf{Trd}}
\newcommand{\eexp}{\msf{exp}}
\newcommand{\llog}{\msf{log}}

\newcommand{\Char}{{\mathrm{char}}}

\newcommand{\mc}{\bm{j}}

\renewcommand{\setminus}{\smallsetminus}

\newcommand{\inject}{\hookrightarrow}
\newcommand{\project}{\twoheadrightarrow}
\newcommand{\too}{\longrightarrow}

\title[Characters of Norm-One Units of Division Algebras]{Characters of the Norm-One Units of Local Division Algebras of Prime Degree}
\author[S. Shechter]{Shai Shechter}\date{\today}
\email{shais@post.bgu.ac.il}
\address{Department of Mathematics\\ Ben Gurion University of the Negev\\ Beer-Sheva 84105\\Israel}
\thanks{
The author is supported by grant 382/11 of the Israeli Science Foundation. Part of this work was completed with the funding of the NRW-Scholarship Program 2014.}
\keywords{Representation theory, division algebras, Representation growth, $p$-adic analytic groups, Representation zeta function}

\begin{document}%\pagewiselinenumbers
\maketitle\setcounter{tocdepth}{2}
\begin{abstract}
We give an explicit construction of all complex continuous irreducible characters of the group $\SL_1(D)$, where $D$ is a division algebra of prime degree $\ell$ over a local field of odd residual characteristic different from $\ell$. For $\ell$ odd, we show that all such characters of $\SL_1(D)$ are induced from linear characters of compact-open subgroups of $\SL_1(D)$. We also compute an explicit formula for the representation zeta function of $\SL_1(D)$.
 \end{abstract}
\tableofcontents
\section{Introduction}

Let $K$ be a local field with ring of integers $\oo$ of odd residual characteristic $p$, and let $D$ be a central division algebra over $K$. Let $\ell$ denote the degree of $D$, that is $\ell=\left(\dim_K(D)\right)^{1/2}$. Let $L/K$ be an extension of degree $\ell$. The field $L$ embeds in $D$ as a maximal subfield and hence there exists an embedding of $D$ into the full matrix algebra $\matr_{\ell}(L)$. The restriction of the determinant map to the image of $D$ in $\matr_\ell(L)$ gives rise to the reduced norm map, denoted by $\Nrd$. This definition is independent of the choice of extension $L/K$ and of the embedding of $L$ into $D$ (see~\cite{pierce}). The valuation map $\val_K:K\to\dbZ\cup\set{\infty}$ extends uniquely to $D$ via the formula \begin{equation}\label{equa:absolutevalue} \val_D(x)=\frac{1}{\ell}\val_K\left(\Nrd(x)\right)\quad(x\in D).\end{equation}

Let $\SL_1(D)$ be the subgroup of elements of reduced norm $1$ in $D^\times$. In this article we focus on the group $\SL_1(D)$ and give an explicit construction of all complex continuous irreducible characters of $\SL_1(D)$, under the assumption that $\ell$ is a prime number different from $p$. This construction enables us to determine the number of such irreducible characters of $\SL_1(D)$ of any given dimension. In particular, we obtain a formula for the representation zeta function of $\SL_1(D)$. 

Given a topological group $\Gamma$ we write $\irr(\Gamma)$ to denote the set of complex continuous irreducible characters of $\Gamma$. The representation zeta function of $\Gamma$ is defined by the Dirichlet generating function
\begin{equation}\label{equation:zeta}\zeta_\Gamma(s):=\sum_{\chi\in \irr(\Gamma)}\chi(1)^{-s}\quad( s\in\dbC).\end{equation}

The abscissa of convergence of $\zeta_\Gamma(s)$ is the infimum of all $\alpha\in\dbR$ such that the series in~\eqref{equation:zeta} converges on the complex right half-plane $\set{s\in\dbC\mid \Re(s)>\alpha}$. %\BBB{The value of the abscissa of convergence of $\zeta_{\Gamma}(s)$ % is given by $\limsup\tfrac{\log(r_n(\Gamma))}{\log(n)}$, and 
%is finite if and only if $\Gamma$ has \textit{polynomial representation growth }(PRG), i.e. if the sequence $r_n(\Gamma)$ is bounded by a polynomial in $n$.

%In the case where $\Gamma$ is an arithmetic group, a fundamental result by Lubotzky and Martin asserts that $\Gamma$ has polynomial representation growth if and only if it has the \textit{congruence subgroup property} (CSP) \cite{LubotzkyMartin}. The representation zeta function of  } 

In the case where $\Gamma$ is a compact $p$-adic analytic group (for $p>2$), Jaikin-Zapirain has shown that the representation zeta function of $\Gamma$ is of the form
\begin{equation}\label{equation:zetarational}
\zeta_\Gamma(s)=\sum_{i=1}^r n_i^{-s} f_i(p),\end{equation}
where $n_1,\ldots, n_r\in\dbN$ and $f_1(t),\ldots,f_r(t)$ are rational functions with integer coefficients (see \cite{jaikin}). Jaikin-Zapirain's proof relies on an interpretation of the representation zeta function of such groups as a $p$-adic integral, based on Howe's version of the Kirillov orbit method for compact $p$-adic analytic groups, and on a model theoretic proof of the rationality of such integrals.

Avni, Klopsch, Onn and Voll considered the case where $\Gamma$ arises as a principal congruence subgroup of the completion of an arithmetic group over a number field% semisimple algebraic group, defined over a number field
, at a non-archimedean prime, and studied the variation of the representation zeta function of $\Gamma$ along the set of non-archimedean places in \cite{AKOVarithmetic}. In particular, it is shown that the representation zeta function of $\Gamma$ can be expressed by a universal formula in the residual cardinality of the ring of integers. Furthermore, the formula given in \cite[Theorem~A]{AKOVarithmetic} is stable under extending the base field. 
%\BBB{
%In this context, the group $\SL_1(D)$ naturally arises as the completion of semisimple algebraic groups of type $\msf{A}_{\ell-1}$, defined  over a number field, at a non-splitting prime (see~\cite[Chapter~9]{PlatonovRapinchuk}). 

%}

In this global context, the group $\SL_1(D)$ naturally arises as the completion of certain arithmetic groups of type $\msf{A}_{\ell-1}$ at a non-archimedean prime (see~\cite[Chapter~9]{PlatonovRapinchuk}). As such, the representation zeta function of $\SL_1(D)$ occurs as constituent in the Euler product decomposition of the representation zeta function of such groups, as described in \cite[Proposition~1.3]{LarsenLubotzky}. Larsen and Lubotzky considered the group $\SL_1(D)$ as a special case in \textit{loc. cit.} and showed an anomalous behaviour of the abscissa of convergence of $\zeta_{\SL_1(D)}(s)$ in comparison to other groups arising as non-archimedian completions of arithmetic groups (viz. isotropic groups). Specifically, it is shown that the abscissa convergence of $\zeta_{\SL_1(D)}(s)$ is $\frac{2}{\ell}$ if $K$ is of characteristic $0$, whereas the abscissa of convergence of an isotropic group is bounded away from $0$ (see~\cite[Theorems~7.1 and 8.1]{LarsenLubotzky}).

We remark that our computation of the representation zeta function shows that the value of the abscissa of convergence of $\zeta_{\SL_1(D)}$ equals $\frac{2}{\ell}$ in positive odd characteristic as well, for division algebras of prime degree. In fact, the representation zeta function of $\SL_1(D)$ depends only on the cardinality of the residue field of $K$ and admits a unique simple pole at $s=\frac{2}{\ell}$.

%The methods applied in this article are based on Jaikin-Zapirain's computation of the representation zeta function of the group $\SL_2(\oo)$, in the case where $\oo$ is a compact discrete valuation ring, as described in \cite{jaikin}.

%Let $K$ is be local field of residual characteristic $p$ and let $D$ be a central division algebra over $K$. Let $\ell$ denote the degree of $D$, that is $\ell=(\dim_K(D))^{1/2}$. Let $L/K$ be the unique (up to isomorphism) unramified extension of $K$ of degree $\ell$. It is well known (see e.g.~\cite[Chapter 17]{pierce}) that $L$ embeds in $D$ as a maximal subfield and hence there exists an embedding of $D$ as a subalgebra of the matrix algebra $\matr_\ell(L)$. The restriction of the determinant map to the image of $D$ in $\matr_\ell(L)$ gives rise the reduced norm of $D$, denoted by $\Nrd$. Similarly, one defines the reduced trace $\Trd$, by restricting the trace map to the image of $D$. The definitions of $\Nrd$ and $\Trd$ are independent of the choice of embedding of $L$ in $D$.

%In this article we focus on the group $\SL_1(D)$ of elements of reduced norm $1$ in the case where $\ell$ is a prime number different form $p$.
From a local point of view, the group $\SL_1(D)$ is also an important object in the study of the multiplicative group $D^\times$ of $D$ (see \cite{NakayamaMatsushima}). The representations of the group $D^\times$ were studied by Corwin and Howe with a view towards the classification of supercuspidal representations of $\GL_\ell(K)$, for arbitrary $\ell$. In \cite{CorwinHowe}, the authors present an explicit bijection between families of irreducible admisible representations of $D^\times$ of a given level and certain representations of $D_1^\times$, where $D_1\subseteq D$ is a sub-division algebra. Using an inductive argument, this method is used in order to parametrize unitary dual of $D^\times$ by means of distinguished characters of subfield extensions $K\subseteq M\subseteq D$. Koch and Zink investigated this subject further and gave a more complete account, describing an explicit bijection between the set of characters of the absolute Galois group of $K$ of degree dividing $\ell$ and essentially all irreducible representations of $D^\times$ with finite image \cite{KochZink}. For additional information, we refer to \cite[Chapter~5]{BushnellFroelich}. A formula for the number $A_{m,f}$ of irreducible representations of $D^\times$ of level $m$, which split into $f$ irreducible constituents upon restriction to the maximal compact subgroup of $D^\times$, was also computed by H. Koch in~\cite[Section~7]{koch}. 
\subsection{Statement of Main Results}
Our computation of the representation zeta function of $\SL_1(D)$ yields the following.

\begin{mainthm}\label{theo:zeta} Let $\ell$ be a prime number. Let $K$ be a non-archimedean local field of residual cardinality $q$ and odd residual characteristic $p$, different from $\ell$. Put $\iota_\ell(q)=\gcd(q-1,\ell)$ to denote the number of $\ell$-th roots of unity in $K$. Let $D$ be a central division algebra of degree $\ell$ over $K$. Then \[\zeta_{\SL_1(D)}(s)=\frac{\frac{q^\ell-1}{q-1}(1-q^{-\binom{\ell}{2}s})+\left(\frac{q^\ell-1}{\iota_\ell(q)\cdot(q-1)}\right)^{-s}\cdot\iota_\ell(q)^2\cdot (q-1)\cdot\left(\sum_{\lambda=0}^{\ell-2}q^{\lambda(1-\frac{\ell-1}{2}\cdot s)}\right)}{1-q^{(\ell-1)-\binom{\ell}{2}s}}.\]
\end{mainthm}

The specific case where $D$ is a quaternion algebra (i.e. $\ell=2$) over a field of characteristic $0$ was computed in \cite{AKOVsome-p}, by applying the machinery of $p$-adic integration for computing the representation zeta function, developed in \cite{AKOVarithmetic}. Using similar methods, the case where $\ell=3$ was computed in the author's M.Sc. thesis.

The valuation on $D$ endows $\SL_1(D)$ with the structure of a profinite group. In particular, it is a totally disconnected compact group, and the kernel of any complex continuous representation of $\SL_1(D)$ must be open in $\SL_1(D)$. We show that there exists a tight connection between the degree of an irreducible character of $\SL_1(D)$ and the 'depth' of the open subgroup on which it trivializes.

To be more precise, let $\OO$ denote the ring of integers of $D$ and $\PP$ be its maximal ideal (see Section~\ref{section:preliminaries}). We fix a system of open subgroups of $\SL_1(D)$ given by ${G^m:=\SL_1(D)\cap (1+\PP^m)},$ for any $m\in\dbN$. The \textit{level} of a complex continuous representation $\rho$ of $\SL_1(D)$ is defined to be the minimal number $m\in\dbN_0$ such that the restriction of $\rho$ to $G^{m+1}$ is trivial. The level of a character $\chi\in \irr(\SL_1(D))$ is defined to be the level of any representation affording it. For every $m\in\dbN_0$ let $\irr^m(\SL_1(D))$ denote the subset of $\irr(\SL_1(D))$ consisting of characters of level $m$.

Given natural numbers $\ell$ and $m$ we define the following integer-valued functions\nopagebreak
\begin{align}
a^\ell_m(t)&:=\begin{cases}\iota^2_\ell(t)\cdot(t-1)t^{m-\ceil{\frac{m}{\ell}}}&\text{ if }\ell\nmid m\\~\\
\frac{t^\ell-1}{t-1}(t^{\ell-1}-1)\cdot t^{(\ell-1)(\frac{m}{\ell}-1)}&\text{ otherwise},
\end{cases}\label{equa:adefin}\\
\intertext{and} \notag~\\
d^\ell_m(t)&:=\begin{cases}\frac{t^\ell-1}{\iota_\ell(t)\cdot (t-1)}t^{\frac{\ell-1}{2}(m-1)}&\text{ if }\ell\nmid m\\~\\
t^{\frac{\ell-1}{2}m}&\text{ othewise},
\end{cases}\label{equa:ddefin}
\end{align}
where $\iota_\ell(t):=\gcd(\ell,t-1)$. We also put $a_0^\ell(t):=\frac{t^\ell-1}{t-1}$ and $d_0^\ell(t):=1$ for any $\ell\in\dbN$.

\begin{mainthm}\label{theo:leveldimension} Let $\ell$ be a fixed prime and let $K$ be a non-archimedean local field of residual cardinality $q$ and odd residual characteristic $p$, different from $\ell$. Let $D$ be a central division algebra of degree $\ell$ over $K$. 

For every $m\in\dbN_0$, the set $\irr^m(\SL_1(D))$ of complex continuous irreducible characters of $\SL_1(D)$ of level $m$ consists of $a_m^\ell(q)$ distinct characters. Furthermore, all such characters are of degree $d_m^\ell(q)$.
\end{mainthm}

The analysis leading up to Theorem~\ref{theo:leveldimension} involves an explicit construction of all irreducible characters of $\SL_1(D)$. Namely, given an irreducible character $\chi$ of $\SL_1(D)$ of level $m$, we consider the irreducible constituents of the restriction of $\chi$ to the group $G^{\ceil{\frac{m+1}{2}}}$, and show that $\chi$ can be recovered by means of extension and induction of such a constituent. In particular, in the case where $\ell\ne 2$ we obtain the following.

\begin{mainthm}\label{theo:monomial}Let $D$ be a division algebra of odd prime degree $\ell$ over a local field $K$ of odd residual characteristic $p$, different form $\ell$. Then all irreducible characters of the group $\SL_1(D)$ are monomial. 
\end{mainthm}

\subsection{Methods and Organization}
\subsubsection{Analysis of Regular Characters} The mechanism applied in this paper for the study of characters of $\SL_1(D)$ is based on the  study of regular characters of $\GL_n(\oo)$, initiated by G. Hill in \cite{Hill} and applied by A. Jaikin-Zapirain in the computation of the representation zeta function of $\SL_2(\oo)$ \cite[Section~7]{jaikin} and by Krakovski, Onn and Singla to the study of the regular representation zeta function of $\GL_n(\oo)$ and $\GU_n(\oo)$ \cite{UriRoiPooja}. 

In the context of $\GL_n(\oo)$, using a finite-type 'Lie-correspondence' between subquotients of $\GL_n(\oo)$ and $\matr_n(\oo)$, one associates to any positive-level character in $\irr(\GL_n(\oo))$  a conjugacy class of matrices over the residue field of $\oo$. Such a character is called \textit{regular} if the corresponding conjugacy class consists of regular matrices, i.e. matrices whose characteristic polynomial and  minimal polynomial coincide;  see~\cite[Definition~1.1]{UriRoiPooja}. 

Viewed as a matrix group over $K$, any non-central element of $D$ generates a field extension of $K$ of degree $\ell$ and in particular has an irreducible minimal polynomial of degree $\ell$, which coincides with its characteristic polynomial. This, along with an analguous version of the Lie-correspondence for $\SL_1(D)$ (described in Section~\ref{section:chatcorres}), makes the group in question amenable to an analysis similar to \cite{Hill}.

\subsubsection{Similarity Classes}
An important ingredient in applying the above-mentioned mechanism is the analysis of the action of the groups $\OO^\times$ and $\SL_1(D)$ on the congruence quotients $D/\PP^m$. This sort of analysis was carried out for division algebras of arbitrary degree by H. Koch in \cite{koch}, where similarity classes in $D$ modulo powers of the prime $\PP$ were shown to be parametrized by means of Eisenstein polynomial sequence associated to the elements of $D$. Koch also obtains a formula for the centralizer modulo $\PP^m$ of an element $y\in D$ \cite[Theorem~9.1]{koch}. In particular, it is shown that the ring of elements of $\OO$, whose image modulo $\PP^m$ commutes with $y+\PP^m$, is additively generated by sets of the form $\CC_{\OO}(y_j)\cap \PP^{m-j}$ ($j\le m$), where  $y_j\in D$ are distinguished elements which are congruent to $y$ modulo $\PP^j$ (viz. minimal elements of $y+\PP^j$), and $\CC_{\OO}(y_j)$ are their centralizers in $\OO$.

The case where $D$ is of prime degree allows for a much simpler analysis, which we present in Section~\ref{section:similarity} (see Proposition~\ref{propo:centralizerformula}). Furthermore, by the Double Centralizer Theorem, the centralizer of a non-central element of $D$ is equal to the field extension of $K$ generated by this element. In particular, this allows us to describe the centralizer modulo $\PP^m$ of $y$ in $\OO^\times$ as a group of the form $\CC_{\OO^\times}(y)\cdot (1+\PP^{m-j})$ where $j$ is determined by $y$, and $\CC_{\OO^\times}(y)$ is an abelian group.

As can be seen in the construction of characters of $\SL_1(D)$ (Section~\ref{section:generalcase}), the above description of centralizers modulo $\PP^m$ supplies us with the tools required in order to describe for the stabilizers of characters of certain congruence subquotients of $D$. The existence of such a description of stabilizers is one of the factors allowing us to extend all such characters to their stabilizers and consequently to construct all characters of $\SL_1(D)$.

\begin{rem} The restriction on the residual characteristic of $\oo$ ($p\ne 2$) is actually superfluous in the description of centralizers modulo $\PP^m$ and only comes into play in the succeeding sections (see Remark~\ref{rem:even-res-char}). It may very well be that this restriction is merely technical, and that the methods applied here can be modulated to account for the case $p=2$ as well.
\end{rem}
\iffalse
\begin{rem} We remark that the description above of centralizers modulo $\PP^m$ holds in the case where the residual characteristic of $\oo$ equals $2$ as well. Consequently, the description of stabilizers of the above mentioned characters is identical in the case where $p=2$. The main obstacle to applying our method in this case lies in a technical difficulty in establishing the Lie correspondence of congruence subquotients of $D$ in Section~\ref{section:chatcorres}. However, as seen in Section~\ref{section:generalcase}, this obstacle only affects our analysis of characters of even level and it is likely that the methods applied here can be modulated to be suitable for handling the case where $p=2$ as well.% However- attempting to do so is beyond the scope of this paper.
\end{rem}
\fi

\subsubsection{Further Research} A natural continuation of the present study would the case where the degree of $D$ is no longer assumed prime. 

An initial observation to be made in this case is that the analysis of centralizers in $D$ becomes significantly more involved. For example, it is often the case the the centralizer of a non-central element in a division algebra of composite degree is a central division algebra over a field extension of $K$, rather than a subfield of $D$. This fact is evident in the added complexity in Koch's centralizer modulo $\PP^m$ formula and is likely to hinder a direct application of our method in this case.

A possible alternative strategy for the construction of characters of $\SL_1(D)$ may be an inductive one, whereby one considers characters arising from characters of subgroups of the form $\SL_1(D_1)$ where $D_1\subseteq D$ is a division algebra of smaller degree whose center is an extension of $K$. Such an approach was utilized by L. Corwin in  \cite{CorwinUnitary} for the parametrization and construction of representations of $1+\PP$ and of $D^\times$ in arbitrary degree.

%It would also be interesting to find out whether our results hold once the restriction on the residual characteristic of $\oo$ and on the degree of $D$ are removed. Regarding the former, we remark that a large portion of the construction we undertake in this paper is independent on the value of $p$, and that the assumption that $p$ is odd only comes into play in the set-up of the Lie-correspondence of $\SL_1(D)$ and in the construction of characters of even degree. Thus, it seems plausible to conjecture that our results hold if $p=2$ as well, so long as $\ell\ne p$. As to the latter restriction- 
%Another possible extensions of our results would be the study of groups of the form $\SL_n(D)$, of $n\times n$ matrices over $D$ with determinant of reduced norm $1$. It may be interesting to consider the variation of the representation zeta function of the groups $\SL_n(D)$ , as the value $n\cdot \deg(D)$ remains constant.

%\subsubsection{Organization} Section~\ref{section:preliminaries} acts as a preliminary section for this paper, in which we recall some basic structural properties of division algebras over local fields, and present the main players of our analysis. In Section~\ref{section:similarity} we present a description of the action of the group of units $\OO^\times$ and of $\SL_1(D)$ on the quotients $D/\PP^m$, reproving Koch's results in this setting. Section~\ref{section:chatcorres} establishes the representation-theoretic foundations for the construction of characters of $\SL_1(D)$, namely 
\subsubsection{Organization} Section~\ref{section:preliminaries} contains preliminary results, including the basic structural properties of division algebras over a local field. In Section~\ref{section:similarity} we consider the action of the multiplicative group $\OO^\times$ and of $\SL_1(D)$ on the quotients $D/\PP^m$ and compute the cardinality of the orbits of this action. Section~\ref{section:chatcorres} gathers the representation-theoretic tools required in the construction- namely, the Lie Correspondence between congruence subquotients of $\SL_1(D)$ and corresponding finite Lie-rings and the method of Heisenberg lifts for $\SL_1(D)$. Finally, Sections~\ref{section:generalcase} and \ref{section:zeta} contain the explicit construction of irreducible characters of $\SL_1(D)$, and the proofs of Theorems~\ref{theo:zeta}, \ref{theo:leveldimension} and \ref{theo:monomial}.
%\subsubsection{Organization} In Section~\ref{section:preliminaries} we review some basic structural properties of division algebras over a local field and present the main players which occur in our analysis. In Section~\ref{section:similarity} we will present an analysis of the similarity classes of the action the group of units $\OO^\times$ and of $\SL_1(D)$ on the quotients $D/\PP^m$ and compute the cardinality of the orbits of this action. In Section~\ref{section:chatcorres} we present the representation-theoretic tools which we require in the construction. Namely, we establish the existence of a finite-type 'Lie Correspondence' between certain sub-quotients of $G$ and corresponding finite Lie-rings, and overview the method of Heisenberg lifts for $\SL_1(D)$, which we utilize in the construction of irreducible characters. Section~\ref{section:generalcase} contains the explicit construction of irreducible characters of $\SL_1(D)$, as well as the proof of Theorem~\ref{theo:monomial}, and the assertion regarding dimensions in Theorem~\ref{theo:leveldimension}. Lastly, in Section~\ref{section:zeta} we collect the results of the previous sections and prove Theorems~\ref{theo:zeta} and \ref{theo:leveldimension}.

\subsection{Notation} 
We denote the ring of integers and maximal ideal of $K$ by $\oo$ and $\pp$ respectively, and fix $\pi$ to be a uniformizer of $\pp$. Let $q=\abs{\oo/\pp}$ and $p=\Char(\dbF_q)$. For any field $K\subseteq M\subseteq D$ we put $\oo_M$ and $\pp_M$ to denote the ring of integers and maximal ideal of $M$ respectively. The letter $L$ will be reserved to denote a fixed subfield of $D$, such that $L/K$ is unramified. The symbol $\iota:=\iota_\ell(q)=\gcd(q-1,\ell)$ denotes the number of $\ell$-th roots of unity in $K$.

The symbol $\val=\val_D$ denotes the valuation on $D$, normalized so that ${\val(K^\times)=\dbZ}$.

We fix the notation $G:=\SL_1(D)$. As already mentioned, for any $m\in\dbN$ we put $G^m=G\cap (1+\PP^m)$, where $\PP$ is the maximal ideal of $\OO$, the ring of integers of $D$. For any $r,m\in\dbN$ with $r\le m$ we write $\GG_m:=G/G^m$ and $\GG^r_m:=G^r/G^m$.

We also use the notation $\mfr{g}:=\mfr{sl}_1(\OO)$ to denote the Lie ring of elements of reduced trace $0$ in $\OO$ (see Section~\ref{subsect:structure}), with the algebra commutator as a Lie-bracket. For any $m\in\dbZ$ write $\mfr{g}^m:=\set{x\in \PP^m\mid \Trd(x)=0}$. Similar to the group case, we write $\g_m:=\mfr{g}/\mfr{g}^m$ and $\g^r_m:=\mfr{g}^r/\mfr{g}^m$, for all $r,m\in\dbZ$ with $r\ge m$.

For any Galois extension $E/F$ we write $\Gal(E/F)$ for the Galois group and  $\Nr_{E/F}$ and $\Tr_{E/F}$ for the field-norm and trace respectively. We put $\SL_1(E\mid F)$ to denote the group of elements of field norm $1$, and  $\mfr{sl}_1(E\mid F)$ to denote the additive group of elements of field trace $0$.

Given a topological group $\Gamma$ we write $\irr(\Gamma)$ to denote the set of continuous complex irreducible characters of $\Gamma$. If no topology is given on $\Gamma$ we simply include all irreducible complex characters of $\Gamma$. Throughout this paper, all representations and characters are assumed to be complex and continuous. For a finite group $\Gamma$ and characters $\chi,\varphi\in \irr(\Gamma)$, we write \[\left[\chi,\varphi\right]_\Gamma:={\abs{\Gamma}}^{-1}\cdot\sum_{\gamma\in\Gamma}\chi(\gamma)
\bar{\varphi(\gamma^)},\] to denote the standard inner-product of characters of $\Gamma$ (see, e.g.~\cite[Definition~2.16]{Isaacs}).

%\color{red}By some abuse of notations we will write $\rho\in\irr(\Gamma)$ to say that $\rho$ is an element of an equivalence class in $\Gamma$ (i.e. that $\rho$ is an irreducible representation of $\Gamma$). We write $\rho\simeq \rho'$ to say that $\rho$ and $\rho'$ are equivalent representations.\color{black}\footnote{@ Uri: I haven't changed anything here- Do you think this needs rewriting?}

Given an abelian group $\Delta$, we write $\Delta^\dual:=\hom_{\msf{Groups}}(\Delta,\dbC^\times)$ for the Pontryagin dual of $\Delta$. If $\Delta$ is equipped with an additional structure (e.g. if $\Delta$ is a Lie-ring), then $\Delta^\dual$ will simply refer to the Pontryagin dual of the underlying additive group.

Given a group $\Gamma$ and elements $g,h\in\Gamma$ we denote the group commutator $ghg^{-1}h^{-1}$ by $(g,h)$. For a Lie-ring $\Delta$ and $x,y\in\Delta$, we denote the Lie bracket of $\Delta$ by $[x,y]$.
	
\subsection*{Acknowledgements}%~\vspace{2cm}%~
%\color{white}
I wish to express my gratitude to Uri Onn for having introduced me to the research question, and for his guidance and many helpful comments in the process of completing this text. I also wish to thank Benjamin Klopsch for hosting and advising part of this work and Christopher Voll for many helpful discussions and remarks.

%\color{black}
\section{Preliminaries} \label{section:preliminaries}
\subsection{The Structure of Local Division Algebras}\label{subsect:structure} The ring of integers $\OO$ of $D$ consists of all elements of non-negative valuation in $D$ and is a non-commutative discrete valuation ring with a unique maximal ideal $\PP$. The ideal $\PP$ is generated by a uniformizing element $\nu\in\PP$ with $\val(\nu)=\frac{1}{\ell}$. Without loss of generality, we may assume that $\nu^\ell=\pi$. 

As $\nu\notin K$, the map $x\mapsto \nu^{-1}x\nu$ is a non-trivial automorphism of $D$. The unique (up to isomorphism) unramified extension $L/K$ of degree $\ell$ can be embedded in $D$ so that the subfield $L$ is invariant under conjugation by $\nu$. Moreover, in this case, the division algebra $D$ is isomorphic to the cyclic algebra $(L,\sigma,\pi)$, where $\sigma$ is a generator of the Galois group $\Gal(L/K)$. That is to say, the following holds (see \cite[Proposition 15.1.a]{pierce}).

\begin{propo}[Structure of Local Division Algebras] \label{propo:structuretheo}Let $D$ be a division algebra over $K$ of degree $\ell$, as above. 
\begin{enumerate}
\item As $K$-vector spaces $D=\bigoplus_{j=0}^{\ell-1} \nu^j L$,
\item For any $d\in L$ we have that $\nu^{-1}d\nu=\sigma(d)$, and
\item $\nu^\ell=\pi$.
\end{enumerate}
\end{propo}

In particular, under the identification of $D$ with $\bigoplus_{j=0}^{\ell-1} \nu^j L$, we obtain an embedding of $D$ into the algebra of $\ell\times\ell$ matrices over $L$ under which an element $x\in D$ is mapped to the operator $y\mapsto x\cdot y$ ($y\in D$). In particular , if we write ${x=x_0+\nu x_1+\ldots+\nu^{\ell-1}x_{\ell-1}}$ ($x_0,\ldots, x_{\ell-1}\in L$), then the matrix corresponding to $x$ with respect to the ordered basis $\left(1,\nu,\ldots,\nu^{\ell-1}\right)$ is
\begin{equation}\label{equa:matrixrepn}
\mat{x_0&\pi\sigma(x_{\ell-1})&\pi\sigma^2(x_{\ell-2})&\ldots&\pi \sigma^{\ell-1}(x_1)\\
	x_1&\sigma(x_0)&\pi \sigma^{2}(x_{\ell-1})&\ldots&\pi\sigma^{\ell-1}(x_2)\\
	x_2&\sigma(x_1)&\sigma^2(x_0)&&\pi\sigma^{\ell-1}(x_3)\\
	\vdots	&\ddots&\ddots&&\vdots\\
	x_{\ell-1}&\sigma(x_{\ell-2})&\sigma^2(x_{\ell-3})&\cdots&\sigma^{\ell-1}(x_0)}.
\end{equation}
The reduced norm and trace are given explicitly by the determinant and trace of this matrix, respectively. In particular, one easily sees from \eqref{equa:matrixrepn} that $\Trd(x)=\Tr_{L/K}(x_0)$, and that $\Nrd(x)\equiv\Nr_{L/K}(x_0)\pmod\pp.$

\begin{lem}\label{lem:norm1+P} The reduced norm of an element $x\in 1+\PP$ lies in $1+\pp$. Furthermore, for any $x\in 1+\PP$ there exists a unique $\xi\in 1+\pp$ such that $\xi^\ell=\Nrd(x)$.
\end{lem}
\begin{proof}
Write $x=x_0+\nu x_1+\ldots+\nu^{\ell-1}x_{\ell-1}$, with $x_0,\ldots,x_{\ell-1}\in L$. Then $x-1\in \PP$ implies that $x_1,\ldots, x_{\ell-1}\in\PP$ and $x_0\in 1+\pp_L$. By \eqref{equa:matrixrepn}, we have that \[\Nrd(x)\equiv\Nr_{L/K}(x_0)\pmod\pp\equiv 1\pmod\pp.\]

For the second assertion, note that since $\ell\ne p$, the map $\xi\mapsto \xi^\ell$ is an automorphism of the pro-$p$ group $1+\pp$. In particular, it is bijective and there exists a unique element $\xi\in1+\pp$ such that $\xi^\ell=\Nrd(\xi)=\Nrd(x).$
\end{proof}
\subsection{Coset Representatives of $\OO/\PP$}
\begin{lem}\label{lem:splittingsequence}Reduction modulo $1+\PP$ induces an exact sequence,
\[1\to 1+\PP\too {\OO^\times}\too\dbF_{q^\ell}^\times\to 1.\]
Furthermore, this sequence has a splitting $\dbF^\times_{q^\ell}\inject {\OO^\times}$, whose image is contained in $\oo_L^\times$ and is given by the set of roots of the polynomial $t^{q^\ell-1}-1$.
\end{lem}
\begin{proof}
Follows from Hensel's Lemma.
\end{proof}
Note that in particular, the image of $\dbF_{q^\ell}^\times$ in $\oo_L^\times$ is invariant under the action of $\Gal(L/K)$.

Let $\TT_L$ be a set of coset representatives of $\OO/\PP$ such that $0\in\TT_L$ and such that $\TT^\times_L:=\TT_L\setminus\set{0}$ is the subgroup of $\oo_L^\times$, given by Lemma~\ref{lem:splittingsequence}. We note that $\TT:=\TT_L\cap K$ is a set of coset representatives for $\oo/\pp$ and $\TT^\times:=\TT\setminus\set{0}$ is isomorphic to $\dbF_{q}^\times$.

Any element $x\in L$ can be uniquely written as a convergent series
\[x=\sum_{j\ge j_0}\pi^j t_j,\]
with $j_0=\val(x)$, $t_j\in \TT_L$ for all $j$'s and $t_{j_0}\ne 0$. Moreover, $x\in K$ if and only if $t_j\in \TT$ for all $j\in\dbN$. By Proposition~\ref{propo:structuretheo}, it follows that any $x\in D$ can be written uniquely as a convergent series 
\begin{equation}\label{equa:nuexpnsion}x=\sum_{j\ge j_0}\nu^j t_j,\end{equation}
where here $j_0=\ell\cdot \val(x)$, $t_j\in \TT$ for all $j$'s and $t_{j_0}\ne 0$. 

In the sequel, we refer to the expansion \eqref{equa:nuexpnsion} as the \textit{$\nu$-expansion of $x$}.

\subsection{Ramified and Unramified Elements}

The structure of subfields within a local division algebra of prime degree allows us to divide the elements of $D$ into three classes, according to the nature of the field extension of $K$ which they generate. Namely, we call an element $y\in D$ \textit{unramified} if $K(y)/K$ is a non-trivial unramified extension and \textit{ramified} if $K(y)/K$ is a non-trivial totally ramified extension. Note that, since the degree of any subfield of $D$ over $K$ must divide $\ell$ (see e.g. \cite[Corollary 13.1.a]{pierce}), it follows that any element which is neither ramified nor unramified is an element of $K$ and hence central in $D$. 

\begin{rem} Note that this type of trichotomy fails if the assumption that $\ell$ is prime is dropped. For a more general classification in the case of arbitrary degree we refer the reader to \cite{koch}.
\end{rem}

The following definition gives a numerical criterion for assigning elements of $D$ into these three classes.
\begin{defi} Let $y\in D$ be given. The \textit{jump of $y$} is defined by
\[\mc(y):=\ell\cdot \sup\set{\val(y-\lambda)\mid\lambda\in K}.\]
\end{defi}
We remark that the value $\mc(y)$ coincides with the first \textit{ramified jump} of $y$, as is defined in \cite[\S~3]{koch}, in the case where $\ell\nmid\mc(y)$. If $y=\sum_{j\ge j_0}\nu^j t_j$ is the $\nu$-expansion of $y$ then $\mc(y)$ equals the minimal index $j$ such that $\nu^j t_j\notin K$. 

%The following Lemma lists some basic properties of $\mc$.
\begin{lem}\label{lem:mcproperties}For $y\in D$ the following hold.
\begin{enumerate}
\item For any $\lambda\in K$,\: $\mc(y+\lambda)=\mc(y)$.
\item For any $g\in D^\times$, $\mc(gyg^{-1})=\mc(y)$.
\item An element $y\in D$ central if and only if $\mc(y)=\infty$.
\item An element $y\in D$ is ramified if and only if $\mc(y)<\infty$ and $\ell\nmid\mc(y)$.
\item An element $y\in D$ is unramified if and only if $\mc(y)<\infty$ and $\ell\mid \mc(y)$. 
\end{enumerate}
\end{lem}
\begin{proof}
Assertions (1),(2) and (3) are clear from the definition. Additionally, note that by the trichotomy discussed above, in order to show the equivalences in (4) and (5), it would suffice to prove the \textit{if} implication in both assertions. We start by proving this implication in (4).

Suppose $\ell\nmid \mc(y)$, and in particular $y\notin K$ and the extension $K(y)/K$ is non-trivial. By definition of $\mc$, there exists an element $\lambda\in K$ such that $\val(y-\lambda)=\frac{1}{\ell}\mc(y)\notin \dbZ$ and hence $K(y)/K$ is ramified. Since $\abs{K(y):K}=\ell$, there are no non-trivial sub-extensions of $K$ in $K(y)$ and hence $K(y)/K$ is totally ramified and $y$ is ramified.

Lastly, to prove (5), note that if $\ell\mid\mc(y)$, then up to subtracting an element of $K$ from $y$ and multiplying by $\pi^{-\val(y)}\in K$, we may assume that $y\in\OO^\times\setminus\oo^\times$. Thus $\ell\mid\mc(y)$ if and only if the reduction modulo $\PP$ of $y$ is not an element of $\dbF_q$ and hence the residual degree $f(K(y)/K)$ is greater then $1$. It follows, since $f(K(y)/K)\cdot e(K(y)/K)=\abs{K(y):K}=\ell$, that $e(K(y)/K)=1$ and hence $y$ is unramified.
\end{proof}

The following is a useful property of unramified elements.
\begin{lem}\label{lem:conjtoL} Let $y\in D$ be unramified. Then there exists $g\in {\OO^\times}$ such that $gyg^{-1}\in L$.
\end{lem}
\begin{proof}
As the extension $K(y)/K$ is unramified of order $\ell$, it is isomorphic over $K$ to $L$. Let $f:K(y)\to L$ be a $K$-isomorphism. By the Skolem-Noether Theorem, there exists $u\in D^\times$ such that $f(x)=uxu^{-1}$ for all $x\in K(y)$. Set $g=\nu^{-\val(u)}u\in {\OO^\times}$. Then \[gyg^{-1}=\nu^{-\val(u)}\left(uyu^{-1}\right)\nu^{\val(u)}=\sigma^{\val(u)}\left(f(y)\right)\in L.\]
\end{proof}

\begin{corol}\label{corol:G1conjtoL} Let $y\in D$ be unramified. Then there exist $h\in G^1$ and $y'\in L$ such that $y=hy'h^{-1}.$
\end{corol}
\begin{proof}
Let $g\in {\OO^\times}$ and $y'\in L$ be such that $y=gy'g^{-1}$, as given by Lemma~\ref{lem:conjtoL}. Let $\sum_{j\ge 0}\nu^j t_j$ be the $\nu$-expansion of $g$. Then $t_0\in \TT_L^\times$ and $gt_0^{-1}\in 1+\PP.$ By Lemma~\ref{lem:norm1+P} there exists an element $\xi\in 1+\pp$ such that $\xi^\ell=\Nrd(gt_0^{-1})$. Put $h:=\xi^{-1} gt_0^{-1}\in (1+\PP)\cap G=G^1$. As $\xi$ is central and $y't_0=t_0y'$, we have that \[hy'h^{-1}=g(t_0^{-1}y't_0)g^{-1}=gy'g^{-1}=y,\] as required.
\end{proof}
\section{Similarity Classes}\label{section:similarity}
%In what follows, we require two main technical tools. The first of which is an explicit correspondence between the set of irreducible characters of the subgroups $G^r$ and orbits under the conjugation action of $G$ inside finite Lie rings which arise as subquotients of $D$. The second is a description of the stabilizers of such characters under the conjugation action of $G$. %The results of this section are required in order to address the latter.

Our main objective in this section is to present a general description of the similarity classes of elements of $D$ under the conjugation action of $\OO^\times$, from which we will collect data regarding the conjugation action of $G$ on subquotients of $D $. Let $\Delta\subseteq D^\times$ be a subgroup. Given $m\in\dbZ$, the \textit{$m$-th similarity class} of an element $y\in D$ under $\Delta$ is the orbit in $D/\PP^m$ of its image modulo $\PP^m$, under the conjugation action of $\Delta$. That is
\begin{equation}\label{equation:simclassdefi}\bO^m_\Delta=\set{ gyg^{-1}+\PP^m\mid g\in \Delta}\subseteq D/\PP^m.\end{equation}
The \textit{$m$-th congruence stabilizer} of $y$ under $\Delta$ in $\OO^\times$ is defined by 
\begin{equation}\label{equation:congstabdefi}\St^m_{\Delta}(y):=\set{g\in {\Delta}\mid gyg^{-1}-y\in \PP^m}.\end{equation}
%Similarly, for any subgroup $\Delta\subseteq \OO^\times$ we define the $m$-th congruence stabilizer of $y$ in $\Delta$ by \[\St_{\Delta}^m(y):=\St_{\OO^\times}^m\cap \Delta,\]and the $m$-th similarity class of $y$ under $\Delta$ as its pre-image under the conjugation action of $\Delta$. 
Of special interest to us are the cases where $\Delta$ is $\OO^\times,G$ or $G^1$.

As already mentioned, the similarity classes of $D$ have been studied in the past by H. Koch in \cite{koch}, where a necessary and sufficient condition for two elements $y$ and $y'$ to lie in the same similarity class is established. Koch's result is based on the method of Eisenstein polynomial sequence and is valid for division algebras of arbitrary degree. Moreover, a formula for the centralizer modulo $\PP^m$ of an element $y\in D$ is presented in \cite[Theorem~9.1]{koch}. 

In the language of Koch, any element $y\in D$ can be associated to a sequence of \textit{minimal elements} $y_j\in D$ ($j\in\dbZ$), which are congruent to $y$ modulo $\PP^j$ and generate the smallest possible field extension of $K$. The centralizer modulo $\PP^m$ of $y$ is then shown to be additively generated by subsets of the form $\CC_{\OO}(y_j)\cap\PP^{m-j}$, where the indices $j$ occurring arise as \textit{ramified jumps} (see \cite[(3.1)]{koch}) of the element $y$.

In the case where $D$ is of prime degree is significantly more tractable, as an element $y\in D$ can have at most one ramified jump in this case. In particular, it enables us to present a much simpler formula for the congruence stabilizers of $y$, whose proof avoids much of the technical difficulties and terminology of \cite{koch} and is presented here for the sake of completeness. We obtain the following.
%A necessary and sufficient condition for $y$ and $y'$ to lie in the same $m$-th similarity class was given by Koch in \cite{koch}, using the method of Eisenstein polynomial sequences, for the case where $D$ is a division algebra of arbitrary degree. Moreover, a formula for the $m$-th similarity class of an element $y\in D$ is computed in Section 9 of \textit{loc. cit.}, in the same generality. 

%The case where $D$ is of prime degree allows for a much simpler computation, which we present here. Namely, we prove the following proposition, which the reader could recognize as a special case of \cite[Theorem~9.1]{koch}.

\begin{propo}\label{propo:centralizerformula} Let $y\in D$ and let $m\in \dbZ$ be given.
\begin{enumerate}
\item If $\mc(y)\ge m$ then $\St^m_{\OO^\times}(y)={\OO^\times}$.
\item If $\mc(y)<m$ then $\St_{\OO^\times}^m(y)=\CC_{\OO^\times}(y)\cdot ({1+\PP}^{m-\mc(y)}).$
\end{enumerate}
Here $\CC_{\OO^\times}(y)$ is the centralizer of $y$ in ${\OO^\times}$.
\end{propo}

%\BBB{
%Proposition~\ref{propo:centralizerformula} affords a rather elementary argument, which eludes use of much of the terminology and technical tools necessary in Koch's proof. In view of this, for the sake of completeness, we present the proof of Proposition~\ref{propo:centralizerformula} here.
%}

We also compute the index of the $m$-th congruence stabilizer to obtain the number of ${\OO^\times}$-conjugates of $y+\PP^m$ in $D/\PP^m$. 
\begin{propo}\label{propo:indexcent}Let $y\in D$ and let $m\in\dbZ$ be given. Assume $\mc(y)<m$. Then
\begin{enumerate}
\item If $y$ is unramified, then the $m$-th similarity class of $y$ under $\OO^\times$ consists of $q^{\ell\left(m-\mc(y)-\ceil{\frac{m-\mc(y)}{\ell}}\right)}$ distinct elements.
\item Otherwise, if $y$ is ramified, then the $m$-th similarity class of $y$ under ${\OO^\times}$ consists of $\frac{q^\ell-1}{q-1}q^{(\ell-1)(m-\mc(y)-1)}$ distinct elements.
\end{enumerate}
\end{propo} 

Finally, we make the transition from similarity classes under ${\OO^\times}$ to similarity classes under $G$. Namely, we show the following.
\begin{propo}\label{propo:centrQtoG}Let $y\in D$ and let $m\in\dbZ$ be given. Assume $\mc(y)<m$.
\begin{enumerate}
\item If $y$ is unramified then the $m$-th similarity classes of $y$ under ${\OO^\times}$ and under $G$ coincide.
\item Otherwise, if $y$ is ramified, then the $m$-th similarity class of $y$ under ${\OO^\times}$ is partitioned into $\iota=\gcd(q-1,\ell)$ distinct $m$-th similarity classes under $G$ of equal size.
\end{enumerate}
\end{propo}

%Propositions~\ref{propo:centralizerformula} and~\ref{propo:indexcent} above can also be deduced from the work of Koch. As our proofs are independent of Koch's work and for the sake of self-containment, we include them in their complete form.
\begin{proof}[Proof of Proposition~\ref{propo:centralizerformula}]
Let $m\in\dbZ$ be arbitrary, and let $y\in D$ be fixed with $\mu:=\mc(y)$.

If $\mu\ge m$ then $y$ is congruent modulo $\PP^m$ to a central element, and in this case $\St_{\OO^\times}^m(y)={\OO^\times}$. Hence assertion (1) of Proposition~\ref{propo:centralizerformula} holds. 

Suppose $\mu<m$. By definition, any element of $\CC_{\OO^\times}(y)$ fixes $y$ by conjugation, and hence fixes its image modulo $\PP^m$, for all $m\in\dbZ$, and the inclusion $\CC_{\OO^\times}(y)\subseteq \St^m_{\OO^\times}(y)$ holds. The inclusion ${1+\PP}^{m-\mu}\subseteq \St_{\OO^\times}^m(y)$ is also immediate, by a routine calculation. Thus, the inclusion $ \CC_{\OO^\times}(y)\cdot (1+\PP^{m-\mu})\subseteq \St_{\OO^\times}^m(y)$ holds. 

To prove the converse inclusion, we divide the proof into two cases, namely ramified and unramified elements, and proceed by induction on $m>\mu$. As already mentioned, since for any $\lambda\in K$ we have that $\St^m_{\OO^\times}(y)=\St_{\OO^\times}^m(y+\lambda)$, we may assume that the $\mu=\ell\cdot\val(y)$.

\subsubsection*{The Case of Ramified Elements} \label{section:ellnmidmu}
Assume $y$ is ramified and hence by Lemma~\ref{lem:mcproperties}(4), $\ell$ does not divide $\mc(y)=\mu$. We first prove the assertion for the case $m=\mu+1$.

Let $y'\in\OO^\times$ be such that $y=y'\nu^\mu$.  Let $g\in \St^{\mu+1}_{\OO^\times}(y)$ be arbitrary, and let $\sum_{j\ge 0}\nu^jt_j$ be its $\nu$-expansion. Then $t_0^{-1} g\in 1+\PP$ and hence $t_0^{-1}g\in \St^{\mu+1}_{\OO^\times}(y)$ and $t_0\in \St^{\mu+1}_{\OO^\times}(y)$ as well. Thus $t_0$ fixes the image of $y$ modulo $\PP^m$ and 
\begin{align*}
y=\nu^{\mu}y'&\equiv t_0(\nu^{\mu}y')t_0^{-1}\pmod{\PP^{\mu+1}}\\
&=\nu^{\mu}\left(\sigma^{\mu}(t_0) y't_0^{-1}\right)\pmod{\PP^{\mu+1}},
\end{align*}
whence $y'\equiv \sigma^{\mu}(t_0)y' t_0^{-1}\pmod\PP$. Since $y'\in\OO^\times$ and ${\OO^\times}/(1+\PP)\cong\dbF_{q^\ell}^\times$ is commutative, we deduce that $t_0$ and $\sigma^\mu(t_0)$ represent the same coset module $\PP$, and hence they are equal. Since $\sigma^\mu$ generates $\Gal(L/K)$ we deduce that $t_0\in K$, and hence \[g=t_0\cdot t_0^{-1}g\in (K\cap {\OO^\times})\cdot(1+\PP)\subseteq \CC_{\OO^\times}(y)\cdot( 1+\PP),\]
proving the case $m=\mu+1.$ 

Proceeding by induction, assume the claim is true for a given $m>\mu$ and let $g\in \St_{\OO^\times}^{m+1}(y)$ be given. Since $\St^{m+1}_{\OO^\times}(y)\subseteq \St_{\OO^\times}^{m}(y)$, by induction hypothesis, there exist $g'\in {1+\PP}^{m-\mu}$ and $s\in \CC_{\OO^\times}(y)$ such that $g=sg'$. Since $s$ centralizes $y$, we have that $g'\in \St^{m+1}_{\OO^\times}(y)$ and it suffice to prove that $g'\in\CC_{\OO^\times}(y)\cdot  (1+\PP^{m+1-\mu})$. If $g'\in {1+\PP}^{m+1-\mu}$ there is nothing to show. We assume henceforth that ${g'\in ({1+\PP}^{m-\mu})\setminus ({1+\PP}^{m+1-\mu})}.$

Let $\alpha,\beta\in\dbZ$ be such that $\ell\alpha+\mu\beta=1$, and put $\tilde{\nu}:=\pi^{\alpha}y^\beta$. Then $\tilde{\nu}$ commutes with $y$ and  $\val(\tilde{\nu})=\frac{1}{\ell}$. Let $h\in \OO^\times$ be such that $g'=1+\tilde{\nu}^{m-\mu} h$, then 
\[g'y-yg'=\tilde{\nu}^{m-\mu}(hy-yh)\equiv 0\pmod{\PP^{m+1}},\]
and hence $h\in \St^{\mu+1}_{\OO^\times}(y).$ By the base of induction, there exists $s'\in \CC_{\OO^\times}(y)$ and $h'\in \OO$ such that $h=s'(1+\tilde{\nu}h')$. Then \[g'=1+\tilde{\nu}^{m-\mu}s'+\tilde{\nu}^{m-\mu+1}s' h',\] and since\nopagebreak
\begin{align*}
(1-\tilde{\nu}^{m-\mu}s') g'&=
(1-\tilde{\nu}^{m-\mu}s') (1+\tilde{\nu}^{m-\mu}s'+\tilde{\nu}^{m-\mu+1}s' h')\\
&\equiv 1-\tilde{\nu}^{2(m-\mu)}{s'}^2\pmod{\PP^{m+1-\mu}}\\
&\equiv 1\pmod{\PP^{m+1-\mu}}
\end{align*}
and $1-\tilde{\nu}^{m-\mu}s'\in \CC_{\OO^\times}(y),$ it follows that $g\in \CC_{\OO^\times}(y)\cdot(1+\PP^{m+1-\mu})$, as wanted.

\subsubsection*{The Case of Unramified Elements}

As before, we assume $\mu=\ell\cdot\val(y)$. By Lemma~\ref{lem:conjtoL}, there exists an element $g\in {\OO^\times}$ such that ${gyg^{-1}\in L}$. Since the groups ${1+\PP}^m$ are normal in ${\OO^\times}$ and since $\St^{m}_{\OO^\times}(gyg^{-1})=g\St^m_{\OO^\times}(y)g^{-1}$ and $\CC_{\OO^\times}(gyg^{-1})=g\CC_{\OO^\times}(y)g^{-1}$, it is enough to prove the assertion for $y\in L$. We write $y=\pi^{\mu/\ell} y'$ with $y'\in \oo_L^\times.$

The base of induction in this case is immediate since the centralizer of $y$ in ${\OO^\times}$ is ${\OO^\times}\cap L=\oo^\times_L$ and \[{\OO^\times}=\oo_L^\times\cdot (1+\PP)\supseteq \St^{\mu+1}_{\OO^\times}(y)\supseteq \CC_{\OO^\times}(y)\cdot (1+\PP).\]

Assume, by induction, that the assertion is true for a given $m>\mu$, and let $g\in \St^{m+1}_{\OO^\times
}(y)\subseteq\St^m_{\OO^\times}(y)$ be given. As in the previous the case, by induction hypothesis, it is enough to prove the assertion under the assumption $g\in {1+\PP}^{m-\mu}$. Using the $\nu$-expansion of $g$, we write $g=1+\nu^{m-\mu}t+\nu^{m+1-\mu} h$,  with $t\in\TT_L$ and $h\in \OO$.

If $\ell$ divides $m$ as well, then we can write $g=1+\pi^{\frac{m-\mu}{\ell}}t+\nu^{m+1-\mu}h$ and then
\[(1-\pi^{\frac{m-\mu}{\ell}}t) g\equiv 1\pmod{\PP^{m+1-\mu}},\]
and hence $g \in \oo_L^\times \cdot (1+\PP^{m+1-\mu})$, and $\St^{m+1}_{\OO^\times}(y)=\CC_{\OO^\times}(y)(1+\PP^{m+1-\mu})$ as wanted.

Otherwise, by $[g,y]\in \PP^{m+1}$, we have that
\[\left[1+\nu^{m-\mu} t+\nu^{m+1-\mu}h,\pi^{\frac{\mu}{\ell}}y'\right]\equiv \pi^{\frac{\mu}{\ell}}\nu^{m-\mu}(ty'-\sigma^{m-\mu}(y')t)\pmod{\PP^{m+1}}\equiv 0\pmod{\PP^{m+1}},\]
and hence $ty'\equiv \sigma^{m-\mu}(y')t\pmod\PP$. We claim that this implies that $t=0$. Otherwise, if $t\ne 0$, by considering the $\nu$-expansion of $y$ and arguing as in the ramified case, we get that $y'\equiv\lambda\pmod\PP$, for some $\lambda\in \TT$, as $\sigma^{m-\mu}$ generates $\Gal(L/K)$. But then \[\mc(y)\ge \val\left(y-\pi^{\frac{\mu}{\ell}}\lambda\right)\ge\mu+1,\]
a contradiction to $\mc(y)=\mu$. Thus, $t=0$ and $g\in {1+\PP}^{m+1-\mu}$, as wanted.

\end{proof}
\begin{proof}[Proof of Proposition~\ref{propo:indexcent}] Having proved Proposition~\ref{propo:centralizerformula}, we can now deduce Proposition~\ref{propo:indexcent}. Fix $m\in\dbZ$ and let $y\in D$ be such that $\mu:=\mc(y)<m$. By Proposition~\ref{propo:centralizerformula} and a standard computation, we have that order of $\bO^m_{\OO^\times}$ is\nopagebreak
\begin{align*}
\abs{{\OO^\times}:\St_{\OO^\times}^m(y)}&=\abs{{\OO^\times}:\CC_{\OO^\times}(y)\cdot (1+\PP^{m-\mu})}\\
&=\abs{{\OO^\times}/({1+\PP}^{m-\mu}):\CC_{\OO^\times}(y)\cdot( {1+\PP}^{m-\mu})/({1+\PP}^{m-\mu})}\\
&=\frac{\abs{{\OO^\times}:({1+\PP}^{m-\mu})}}{\abs{\CC_{\OO^\times}(y):\CC_{\OO^\times}(y)\cap ({1+\PP}^{m-\mu})}}&%&\text{(by the Isomorphism Theorems).} 
\end{align*}

One easily verifies that $\abs{{\OO^\times}:({1+\PP}^{m-\mu})}=\frac{q^\ell-1}{q^\ell}q^{\ell(m-\mu)}.$ Additionally, as $\CC_{\OO^\times}(y)$ is simply the intersection of ${\OO^\times}$ with the centralizer of $y$ in $D$, and $y$ is assumed non-central, we have that $\CC_{\OO^\times}(y)$ is the group of units in the ring of integers $\oo_{K(y)}$ of $K(y)$.

Assuming $K(y)/K$ is unramified, by Lemma~\ref{lem:conjtoL}, we have that
\[\abs{\CC_{\OO^\times}(y):\CC_{\OO^\times}(y)\cap( {1+\PP}^{m-\mu})}=\abs{\oo_L^\times:\left(1+\pp_L^{\ceil{\frac{m-\mu}{\ell}}}\right)}=\frac{q^\ell-1}{q^\ell}q^{\ell\ceil{\frac{m-\mu}{\ell}}}.\]

Otherwise, if $K(y)/K$ is totally ramified, we have that $\oo_{K(y)}$ is generated as an $\oo$-module by a uniformizer of $\PP$ (e.g. by the element $\tilde{\nu}$ in the proof of Proposition~\ref{propo:centralizerformula}). Consequently
\[\abs{\CC_{\OO^\times}(y):\CC_{\OO^\times}(y)\cap ({1+\PP}^{m-\mu})}=\abs{\oo_{K(y)}^\times:\left(1+\pp_{K(y)}^{m-\mu}\right)}=\frac{q-1}{q}q^{m-\mu}.\]

\end{proof}
%\subsection{Computation of Order of Similarity Classes Under $\SL_1(D)$}
%Our interest in this Subsection is to prove Proposition~\ref{propo:centrQtoG} and apply the computations of the previous Subsection in order to compute order of the $m$-th similarity class of an element $y\in D$ under $G$. %Let us introduce some notation that will be of use in this proof. Given a subgroup $\Delta\subseteq {\OO^\times}$, put \[\bO_\Delta^m(y):=\set{hyh^{-1}+\PP^m\mid h\in \Delta}, \]
%to denote the orbit of the congruence class of $y$ modulo $\PP^m$ under $\Delta$.

%Our objective is to compute the quantity ${\abs{\bO^m_{\OO^\times}(y)}}/{\abs{\bO^m_G(y)}}$, for a fixed $m$ and $y\in D$ such that $\mu=\mc(y)<m$.

%In actuality, there is another group which makes computations significantly easier for us, which is the group $G^1=G\cap( 1+\PP)$. 

Lastly, we prove Proposition~\ref{propo:centrQtoG}. As it turns out, in order to understand the relation between the similarity class of an element under $\OO^\times$ and under $G$, it is useful to consider the similarity class of the element under the group $G^1$ as well.

Given an unramified element $y\in D$, by invoking Lemma~\ref{lem:norm1+P}, we will prove the inclusion $\bO_{\OO^\times}^m(y)\subseteq\bO_{G^1}^m(y)$ for all $m\in\dbZ$, from which the first assertion of Proposition~\ref{propo:centrQtoG} follows. For the second assertion, in the case where $y\in D$ is a ramified element, we deduce the ratio between the orders of $\bO^m_{\OO^\times}(y)$ and $\bO^m_G(y)$ from Lemma~\ref{lem:conjclassGtoG1} and \ref{lem:conjclassQtoG1} below.

\begin{lem}\label{lem:conjclassGtoG1} Let $y\in D$ be a ramified element and let $m\in\dbZ$ be greater than $\mc(y)$. Then
\[\frac{\abs{\bO_G^m(y)}}{\abs{\bO_{G^1}^m(y)}}=\abs{\SL_1\left(\dbF_{q^\ell}\mid\dbF_q\right):\SL_1\left(\dbF_{q^\ell}\mid\dbF_q\right)\cap\dbF_q^\times}=\frac{q^\ell-1}{\iota\cdot (q-1)}.\]
\end{lem}

\begin{proof} Consider the restriction of the reduction map ${\OO^\times}\project \dbF_{q^\ell}^\times$ to $\St_G^m(y)$. By definition, the kernel of this map is $\St^m_G(y)\cap (1+\PP)=\St^m_{G^1}(y)$. We claim that the image of this map is $\SL_1(\dbF_{q^\ell}\mid\dbF_q)\cap \dbF_q^\times$.  

Since taking norms commutes with the reduction map, the image of this map is included in $\SL_1(\dbF_{q^\ell}\mid\dbF_q)$. Let us show that the image is also included in $\dbF_q^\times$. 

As before, we may assume that $\mu:=\mc(y)=\ell\val(y)$, and write $y=\nu^\mu y'$ with ${y'\in \OO^\times}$. Let $g\in \St_G^m(y)$ be given, with $g=\sum_{j\ge 0}\nu^jt_j$ its $\nu$-expansion. Then the image of $g$ in $\SL_1(\dbF_{q^\ell}\mid\dbF_q)$ equals the image of $t_0$. By assumption, we have that  $gy\equiv yg\pmod\PP^m$, and since $\mu<m$ it follows that $gy=g\nu^{\mu}y'\equiv \nu^\mu y'g\pmod{\PP^{\mu+1}}$, and hence
\[t_0\nu^\mu y'=\nu^\mu\sigma^\mu(t_0) y'\equiv \nu^\mu t_0 y'\pmod{\PP^{\mu+1}}.\]

Since $y\in\OO^\times$, this implies that $\sigma^\mu(t_0)\equiv t_0\pmod\PP$, and hence $t_0$ is invariant under $\Gal(L/K)$. Thus, $t_0\in K$ and the image of $g$ modulo $1+\PP$ is in $\dbF_q^\times$.

Lastly, note that for any $\mbf{t}\in\SL_1(\dbF_{q^\ell}\mid\dbF_q)\cap \dbF_q^\times$, the corresponding coset representative $t\in \TT^\times$ is central in $D$ and in particular an element of $\St^m_G(y)$. Thus, the reduction map is onto $\SL_1(\dbF_{q^\ell}\mid\dbF_q)\cap \dbF_q^\times$.

%Thus, we have an exact sequence 
%\[1\to\St^m_{G^1}(y)\too\St_G^m(y)\too \SL_1(\dbF_{q^\ell}\mid\dbF_q)\cap \dbF_q^\times\to 1.\]

It follows that $\abs{\St^m_{G}(y):\St^m_{G^1}(y)}=\abs{\SL_1\left(\dbF_{q^\ell}\mid\dbF_q\right)\cap\dbF_q^\times}=\iota$ and hence \[\frac{\abs{\bO_{G}^m(y)}}{\abs{\bO_{G^1}(y)}}=\frac{\abs{G:\St_{G}^m(y)}}{\abs{G^1:\St_{G^1}^m(y)}}=\frac{\abs{G:G^1}}{\abs{\St_G^m(y):\St_{G^1}^m(y)}}=\frac{q^\ell-1}{\iota\cdot(q-1)}.\]
\iffalse
\begin{equation}\label{dimonddiagram}
\begin{tikzpicture}[description/.style={fill=white,inner sep=2pt}]
\matrix (m) [matrix of math nodes, row sep=1em,
column sep=1.5em, text height=1.5ex, text depth=0.25ex]
{&G\\ \St^m_G(y) & & G^1 \\
& \St^m_{G^1}(y)& \\ };
%\draw[double,double distance=5pt] (m-1-1) – (m-1-3);
\path[-,font=\scriptsize]
(m-2-1) edge[thick] node[auto] {} (m-1-2)
(m-1-2) edge[thick] node[auto] {} (m-2-3)
(m-2-3) edge[thick] node[auto] {} (m-3-2)
(m-2-1) edge[thick] node[auto] {} (m-3-2);
\end{tikzpicture}
\end{equation}
\fi
\end{proof}

\begin{lem}\label{lem:conjclassQtoG1} Let $y\in D$ be a ramified element and let $m\in\dbZ$ be greater than $\mc(y)$. Then
\[\frac{\abs{\bO^m_{\OO^\times}(y)}}{\abs{\bO^m_{G^1}(y)}}=\abs{\SL_1\left(\dbF_{q^\ell}\mid\dbF_q\right)}=\frac{q^\ell-1}{q-1}.\]
\end{lem}
\begin{proof}Consider the orbit of $y+\PP^m$ under the action of the group $\TT_L^\times$,
\[\bO^m_{\TT^\times_L}(y)=\set{tyt^{-1}+\PP^m\mid t\in \TT_L^\times}\subseteq\bO_{\OO^\times}^m(y).\]
We claim the following assertions from which Lemma~\ref{lem:conjclassQtoG1} easily follows.
\begin{enumerate}
\item The orbit $\bO^m_{\TT^\times_L}$ consists of $\frac{q^\ell-1}{q-1}$ elements.
\item Let $y'\in \bO^m_{\TT_L^\times}(y)\cap\bO^m_{G^1}(y)$. Then $y'\equiv y\pmod{\PP^m}$. 
\item Suppose $x\in \bO^m_{{\OO^\times}}(y)$. Then there exists $y'\in D$ such that $x\in\bO^m_{G^1}(y')$ and $y'+\PP^m\in \bO^m_{\TT_L^\times}(y)$.
\end{enumerate}
As before, it is enough to prove the lemma for the case where $\mu=\mc(y)=\ell\val(y)$. Put $y=\nu^{\mu} y'$ with $y'\in \OO^\times$. 

Suppose $t\in \TT^\times_L$ fixes $y$. Then, similarly to the proof of Lemma~\ref{lem:conjclassGtoG1}, we deduce that $t=\sigma^{\mu}(t)$ and hence $t\in \TT^\times$. As any element of $\TT^\times$ centralizes $y$, by the Orbit-Stabilizer Theorem we obtain that \[\abs{\bO^m_{\TT^\times_L}(y)}=\abs{\TT_L^\times:\TT^\times}=\frac{q^\ell-1}{q-1},\] which proves (1).

Additionally, assertion (3) is an immediate consequence of the decomposition ${\OO^\times}=\TT_L^\times\ltimes (1+\PP)$ of Lemma~\ref{lem:splittingsequence} and Corollary~\ref{corol:G1conjtoL}. Indeed, if $g\in {\OO^\times}$ is such that $gyg^{-1}\equiv x\pmod{\PP^m}$, take $t\in \TT_L^\times$ and $\xi\in 1+\pp$ such that $gt^{-1}\in 1+\PP$ and $\xi^\ell=\Nrd(gt^{-1})$. Then for $y'=tyt^{-1}$ and $g':=\xi^{-1}gt^{-1}\in G^1$ we have that $y'+\PP^m\in \bO^m_{\TT^\times_L}(y)$, and 
\[g'y'{g'}^{-1}=\xi^{-1}(gt^{-1})y'(gt^{-1})^{-1}\xi= gyg^{-1}\equiv x\pmod{\PP^m}.\] Thus $x\in\bO_{G^1}^m(y')$.

Finally, we prove (2). Let $t\in \TT_L^\times$ and $g\in G^1$ be such that 
\[y'\equiv tyt^{-1}\pmod{\PP^m}\equiv gyg^{-1}\pmod{\PP^m}.\]
Then in particular $gt^{-1}$ fixes $y$ modulo $\PP^m$, i.e. $gt^{-1}\in \St^m_{\OO^\times}(y).$ By Proposition~\ref{propo:centralizerformula} and since $\mu<m$ and $g\in 1+\PP$, we get that the image of $t$ modulo $1+\PP$ is in $\CC_{\OO^\times}(y)\cdot(1+\PP)/(1+\PP)$. Since $t$ is a coset representative of ${\OO^\times/(1+\PP)}$, this implies that $t$ centralizes $y$ and hence $y'\equiv y\pmod{\PP^m},$ as wanted.\end{proof}

\begin{proof}[Proof of Proposition~\ref{propo:centrQtoG}]

%\begin{lem}\label{lem:ellsameconj} Let $y\in D$ be an unramified element. Then
%\[\bO_{G^1}^m(y)=\bO^m_G(y)=\bO^m_{\OO^\times}(y).\]
%\end{lem}

\begin{enumerate}
\item Let $y\in D$ be an unramified element. We prove the equality
\[\bO^m_{G^1}(y)=\bO^m_G(y)=\bO^m_{\OO^\times}(y).\]

As the inclusion 
$\bO_{G^1}^m(y)\subseteq\bO^m_G(y)\subseteq\bO^m_{\OO^\times}(y)$ is clear from the definition \eqref{equation:simclassdefi}, we only need to show that $\bO^m_{\OO^\times}(y)\subseteq \bO^m_{G^1}(y)$. By Corollary~\ref{corol:G1conjtoL}, without loss of generality, we may assume that $y\in L$.

Let $y'\in \bO_{\OO^\times}^m(y)$, and let $g\in {\OO^\times}$ be such that $gy'g^{-1}\equiv y\pmod{\PP^m}$. % and let $\tilde{y'}\in D$ be a lift of $y$ (i.e. such that $y'=\tilde{y'}+\PP^m$). 
By considering the $\nu$-expansion of $g$, there exists $t\in \TT_L^\times$ such that $t^{-1}g\in 1+\PP$. By Lemma~\ref{lem:norm1+P} there exists $\xi\in 1+\pp$ such that $\xi^\ell=\Nrd(t^{-1}g)$ and thus $g':=\xi^{-1}t^{-1}g\in (1+\PP)\cap G=G^1.$ Additionally, \[{g'}^{-1}yg'=g^{-1}(tyt^{-1})g=g^{-1}yg\equiv y'\pmod{\PP^m},\]
since $\xi$ is central in $D$ and $y,t\in L$. Thus, $y'\in\bO^m_{G^1}(y)$.
\item 
Let $y\in D$ be a ramified element. The $m$-th similarity class  $\bO^m_{\OO^\times}(y)$ is partitioned into a disjoint union of similarity classes of $\OO^\times$-conjugates of $y$ under $G$. By Lemmas~\ref{lem:conjclassGtoG1} and \ref{lem:conjclassQtoG1}, the number of such similarity classes is
\[\frac{\abs{\bO^m_{\OO^\times}(y)}}{\abs{\bO^m_G(y)}}={\frac{\abs{\bO^m_{\OO^\times}(y)}}{\abs{\bO^m_{G^1}(y)}}}\cdot{\frac{\abs{\bO^m_{G^1}(y)}}{\abs{\bO^m_{G}(y)}}}=\iota.\]
\end{enumerate}
\end{proof}

\section{Character Correspondence for Subquotients of $\SL_1(D)$}\label{section:chatcorres}

Fix $r,m\in\dbN$, such that $r\le m\le 3r$. %Our purpose in this section is to develop some of the tools which will be used in the analysis of the irreducible representations of $\GG^r_m$ in several key cases for our construction. 
Note that in general, the group $\GG^r_m$ is a class-2 nilpotent group, and is abelian whenever $m\le 2r$. Furthermore, in the case where $r$ is divisible by $\ell$ we also have that $\GG^r_{2r+1}$ is abelian, as the commutator subgroup of $G^r$ is included in $G^{2r+1}$ in this case.

In the following section we construct an explicit $G$-equivariant bijection between the group $\GG^r_m$ and the finite Lie-ring $\g^{r}_{m}$. In particular, in the case where $r\le m\le 2r$, this bijection induces a $G$-equivariant identification of the set $\irr(\GG^r_m)$ with a finite Lie-ring of the form $\g^{r'}_{m'}$, for suitable $r',m'\in\dbZ$ (see Corollary~\ref{corol:dualforabelian}). In Section~\ref{subsect:heisenberg} we describe the method of \textit{Heisenberg lifts} which provides us with the tools to utilize this bijection in order to describe the set of irreducible characters of non-abelian groups of the form $\GG^{r}_{2r+1}$.

\subsection{Lie Correspondence}~\label{subsection:liecorres}
\subsubsection*{Truncated Exponential and Logarithm}\label{subsection:explog}
\begin{lem}\label{lem:exp}Let $r,m\in \dbN$ be given with $r\le m\le 3r$. The truncated exponential map, defined by
\begin{align}\eexp\colon \g^{r}_{m}&\too\GG^r_m\notag\\~
X+\mfr{g}_{m}&\mapsto \left(1+X+\tfrac{1}{2}X^2 \right)G^m,\quad (X\in\mfr{g}^r)\notag,\end{align}
%\begin{align*}\log^m\colon \GG^{r}_{m}&\too\g^r_m\\~\\
%(1+X)G^{m}&\mapsto X-\frac{X^2}{2}+\mfr{g}^m,\quad (1+X\in G^r)\end{align*}
establishes a $G$-equivariant bijection between the group $\GG^r_m$ and the Lie-ring $\g^{r}_{m}$. Furthermore, if $m\le 2r$ then this map is a $G$-equivariant isomorphism of abelian groups.

The inverse map of $\eexp$ is given by the truncated logarithm map \[\llog\left((1+X)G^m\right)=X-\tfrac{1}{2}X^2.\]
\end{lem}

%\footnote{@Uri: I don't thin I understand your mutatis mutandis remark. I don't think I ever saw an actual proof of this fact written in a paper}. We remark, however, that the maps $\eexp$ and $\llog$, as defined above, may fail to be inverses if one does not assume $p\ne 2$. For example, if $\Char(K)=0$ and $\val(2)=1$, then %for $r=1,\:m=1$ and $X=\nu$ we have that
 
%\[\eexp(\llog(1+\nu))=1+\nu-\tfrac{1}{2}\nu^3\not\equiv 1+\nu(\mod\PP^3).\]

\iffalse We also remark that by explicit computation, based on \eqref{equa:matrixrepn}, one may verify the following trace identity-
\begin{equation}
\NNrd_m(\eexp(X))=1+\TTrd_m(X)+\tfrac{1}{2}\TTrd_m(X)^2,\quad X\in\g^r_m
\end{equation}
where $\TTrd_m:\OO/\PP^m\to \oo/\pp^{\ceil{m/\ell}}$ is induced from the reduced trace map.
\fi

We also have the following formulas.
\begin{lem}\label{lem:formulasexplog}Let $r,m\in\dbN$ be given such that $r\le m\le 3r$. Then, for any $x,y\in\g^r_m$, we have
\begin{equation}\label{equation:expAdlogad}\llog\left(\left(\eexp(x),\eexp(y)\right)\right)=\left[x,y	\right],\end{equation}
and the truncated Baker-Campbell-Haudorff formula holds
\begin{equation}\label{equation:bch}\llog(\eexp(x)\cdot\eexp(y))=x+y+\tfrac{1}{2}[x,y].\end{equation}
\end{lem}

Here $\left[x,y\right]=xy-yx$ denotes the Lie-bracket on $\g^r_m$, inherited from the associative multiplication on $\OO$, and $\left(g,h\right)=ghg^{-1}h^{-1}$ is the group commutator, for $g=\eexp(x),\:h=\eexp(y)\in G^r_m.$

\begin{rem}\label{rem:even-res-char} The definition of the truncated exponential and logarithm maps necessitates the restriction $p\ne 2$. Nevertheless, in the case where $r\le m\le 2r$, a $G$-equivariant isomorphism of $\g^r_m$ onto $\GG^r_m$ is given by $X+\mfr{g}_m\mapsto (1+X)G^m$, regardless of the residual characteristic of $\oo$. This isomorphism supplies us with a, possibly somewhat limited, Lie-correspondence in the case where $p=2$. 

As can be seen in Subsections \ref{subsect:oddlevel} and \ref{subsect:divby2ell}, much of the construction of characters of $G$ carries through within this limited setting and can be applied in even residual characteristic as well. The construction in the complementary case (Subsection~\ref{subsect:evenlevel}), however, requires invoking the assumption $p\ne 2$ at several key points (namely, the above correspondence and Proposition~\ref{propo:radicals}) and it would be of interest to see if an argument bypassing this restriction can be obtained using an alternative method. 
\end{rem}

\subsubsection*{The Dual of Subquotients of $\mfr{g}$}

The Pontryagin dual of the abelian groups $\g^r_m$, for $r\le m$, can be explicitly described as subquotients of the Lie-algebra $\mfr{g}$. We recall that the Lie-algebra $\mfr{sl}_1(D)$ of traceless elements in $D$ is equipped with a non-degenerate bilinear-form (viz. its Killing form), and hence is self-dual as a $K$-vector space. This isomorphism is equivariant with respect to the conjugation action of $G$. 

%As the finite Lie rings $\g^r_m$ bare a conspicuous resemblance to $\mfr{sl}_1(D)$ one might find it plausible to expect a similar form to induce a $G$-equivariant isomorphism of $\g^r_m$ with its Pontryagin dual.
In the case of the finite Lie-rings $\g^r_m$, a $G$-equivariant isomorphism of the ring with its Pontryagin dual does not necessarily exists.  Nonetheless, as we will see shortly, it is still possible to identify the dual of $\g^r_m$ with a Lie-ring of the form $\g^{r'}_{m'}$, where $r',m'$ are determined by $r$ and $m$. 

\newcommand{\Res}{\mathrm{\varrho}}
We begin by selecting a character of $K$ which factors through $K/\oo$ and does not vanish on $\pi^{-1}$. 
Let $\mfr{k}$ be either $\dbQ_p$ if $\Char(K)=0$ or $\dbF_p\dpar{t}$ otherwise, and let $\mfr{R}$ denote its ring of integers. Define $\psi_\mfr{k}$ by the  composition 
\begin{align*}\psi_{\dbQ_p}&\colon {\dbQ_p\project\dbQ_p/\dbZ_p\xrightarrow{\simeq} \Bmu_{p,\infty}}\inject\dbC^\times ,\\
\intertext{if $\Char(K)=0$, or}
 \psi_{\dbF_p\dpar{t}}&\colon \dbF_p\dpar{t}\xrightarrow{\Res_{0}}\dbF_p\xrightarrow{\simeq}\Bmu_p\inject\dbC^\times,\end{align*}
where $\Res_0$ maps a series $\sum_{j}a_jt^j$, with $a_j\in\dbF_q$
and $a_j=0$ for all $j<j_0$ for some $j_0<0$, to the coefficient $a_{-1}$. Let $\delta:=\delta_{K/\mfr{k}}\in\dbN_0$ be the differential exponent of $K/\mfr{k}$, i.e. such that $\pi^{-\delta}$ generates the inverse different $\mfr{d}_{K/\mfr{k}}^{-1}:=\set{x\in \oo\mid \Tr_{K/\mfr{k}}(x\oo)\subseteq\mfr{R}}.$ Note that by definition $\Tr_{K/\mfr{k}}(\pi^{-\delta-1})\in p^{-1}\mfr{R}\setminus\mfr{R}$ and hence $\Tr_{K/\mfr{k}}(\pi^{-\delta-1})\notin\ker(\psi_{\mfr{k}}).$ The map $\Psi:K\to\dbC^\times$, defined by \begin{equation}\label{equation:Psi-definition}\Psi(x):=\psi_{\mfr{k}}\circ\Tr_{K/\mfr{k}}(\pi^{-\delta-1}\cdot x),\end{equation} is the required character.% holds.\footnote{This $\lambda$ is superfluous, but i think it's easier to add it than to explain why it can be taken to be $1$}

\begin{lem}\label{lem:dualpairing}Let $r,m\in\dbN$ be given with $r\le m$. Define a map
\begin{align}\notag \gen{\cdot,\cdot}^r_m&\colon\g^r_m\times\g^{-m+1}_{-r+1}\too\dbC^\times\\
\gen{x,y}^r_m&:=\Psi\left(\pi^{-1}\cdot \Trd \left(X\cdot Y\right)\right),\label{equation:dualformdefi}
\end{align}
where $X\in\mfr{g}^r$ and $Y\in \mfr{g}^{-m+1}$ are lifts of $x$ and $y$ respectively. Then $\gen{\cdot,\cdot}^r_m$ is a well-defined, $G$-invariant and non-degenerate bilinear pairing.  
\end{lem}
\begin{proof}

%We sketch the proof of Lemma~\ref{lem:dualpairing} and refer the reader to Appendix~\ref{appendix:proofdualpairing} for a complete proof. 
First note that for any $k\in\dbZ$ we have that $\Trd(\PP^k)\subseteq\pp^{\ceil{k/\ell}},$ as can be seen in \eqref{equa:matrixrepn}. In particular, we have that $\val\left(\pi^{-1}\Trd(XY)\right)\ge 0$ whenever $X\in \mfr{g}^m$ and $Y\in\mfr{g}^{-m+1}$ or $X\in \mfr{g}^r$ and $Y\in\mfr{g}^{-r+1}$. The unambiguity of the definition of $\gen{\cdot,\cdot}^r_m$ follows from this, along with the fact that $\Psi$ vanishes on $\oo$ (see~\eqref{equation:Psi-definition}). 

To prove non-degeneracy, let $x\in \g^r_m$ be non-zero and pick a lift $X\in\mfr{g}^r\setminus\mfr{g}^m$. Consider the element $Y'=\frac{1}{\ell}X^{-1}. $ Since $\ell\notin\PP$ and $\val(X)\le\frac{m}{\ell}$ we have that $Y'\in \PP^{-m+1}$. Taking $Y=\left(Y'-\frac{1}{\ell}\Trd(Y')\right)\in \mfr{g}^{-m+1}$ we have that
\[\pi^{-1}\Trd(XY)=\pi^{-1}\Trd(XY')=\pi^{-1}\notin\ker(\Psi).\]
Non-degeneracy follows by picking $y=Y+\mfr{g}^{-r+1}$. 

The $G$-invariance and bilinearity  of $\gen{\cdot,\cdot}^r_m$ are clear from the definition.
\end{proof}

\begin{defi}\label{defi:dualitymap} Given $r,m\in\dbN$ and $y\in \g^{-m+1}_{-r+1}$ we define the map $\phi_y^{r,m}\in(\g^r_m)^{\dual}$ by
\[\phi_y^{r,m}(x)=\gen{x,y}^r_m.\]
In order to lighten up notation, whenever the indices $r$ and $m$ are clear from context we will omit them from the notation and simply write $\phi_y=\phi_y^{r,m}$.
\end{defi}
\begin{propo}\label{propo:dual} Let $r,m\in\dbN$ be given with $r\le m$. Then, the map $y\mapsto \phi_y$ defined on $\g_{-r+1}^{-m+1}$ is a $G$-equivariant isomorphism of the group $\g^{-m+1}_{-r+1}$ onto the Pontryagin dual of $\g^r_m$, which intertwines the coadjoint action of $G$ on $\left(\g^r_m\right)^\dual$ and the adjoint action on $\g^{-m+1}_{-r+1}.$
\end{propo}

\begin{proof}
%Given $r,m\in\dbN$ with $r\le m$ and $y\in\g^{-m+1}_{-r+1}$ we put $\phi_y$ to denote the map defined on $\g^r_m$ by \[\phi_y(x)=\gen{x,y}^r_m,\quad(x\in\g^r_m).\] 
By Lemma~\ref{lem:dualpairing} the correspondence $y\mapsto \phi_y$ establishes a $G$-equivariant injection of $\g^{-m}_{-r+1}$ into the Pontryagin dual of $\g^r_m$. The surjectivity of the map follows since 
\[\abs{(\g^r_m)^\dual}=\abs{\g^r_m}=\abs{\g^{-m+1}_{-r+1}}.\]
The last equality follows from $\abs{\g^r_m}=q^{\ell(m-r)-\left(\ceil{\frac{m}{\ell}}-\ceil{\frac{r}{\ell}}\right)}$ for any $r,m\in\dbZ$, and the general equality $\ceil{\frac{-m+1}{\ell}}=-\ceil{\frac{m}{\ell}}+1$, which hold for all $m\in\dbZ$.
\end{proof}

In particular, by Lemma~\ref{lem:exp}, we have the following.
\begin{corol}\label{corol:dualforabelian} Let $r,m\in\dbN$ be such that $r\le m\le 2r$ an let $\vartheta\in\irr(\GG^r_m)$ be given. Then there exists a unique element $y_\vartheta\in\g^{-m+1}_{-r+1}$ such that $\vartheta=\llog^*(\phi_{y_\vartheta})$. Moreover, the correspondence $\vartheta\mapsto y_{\vartheta}$ is a $G$-equivariant bijection of $\irr(\GG^r_m)$ with $\g^{-m+1}_{-r+1}$.
\end{corol}

Here $\llog^*:(\g^r_m)^\dual\to\irr(\GG^r_m)$ stands for pre-composition with $\llog$, i.e. $\llog^*(\phi)(g)=\phi\left(\llog(g)\right)$, for $g\in \GG^r_m$ and $\phi\in(\g^r_m)^\dual$.

\subsection{Heisenberg Lifts for Subquotients of $\SL_1(D)$}\label{subsect:heisenberg}

In this subsection, we limit our description of the method of Heisenberg lifts to the study of irreducible characters of non-abelian subquotients of $\SL_1(D)$ of the form $\GG^r_{2r+1}$. For the more general description, we refer the reader to \cite{BushnellFroelich}, for a general treatment, and \cite[\S~3.2]{UriRoiPooja}, for the case of regular representations of $\GL_n(\oo)$ and $\GU_n(\oo)$.

\newcommand{\ggama}{\bm\upgamma}%In this Subsection we will describe a method which enables us to obtain information regarding the representation theory of a finite nilpotent $p$-group of nilpotency class $2$ from representations of a central subgroup, under some conditions which we describe below. 
We fix $r\in\dbN$ which is not divisible by $\ell$ and put $\Gamma^1:=\GG^{r}_{2r+1}$ and $\Gamma^2:=\GG^{r+1}_{2r+1}$. Note that $\Gamma^2$ is central in $\Gamma^1$, and that $\mfr{f}:=\Gamma^1/\Gamma^2$ is abelian. In fact, since $\ell\nmid r$, one has that $\mfr{f}$ is isomorphic to the additive group of $\dbF_{q^\ell}$ (see \cite[Theorem~7]{riehm}).

Let $\ggama_1:=\g^r_{2r+1}$ and $\ggama_2:=\g^{r+1}_{2r+1}$. In this setting, the truncated exponential map $\eexp:\ggama_1\to \Gamma^1$ is a $G$-equivariant bijection, whose restriction to $\ggama_2$ is an isomorphism onto $\Gamma^2$ (by Lemma~\ref{lem:exp}). Furthermore, this map induces an isomorphism of $\ggama_1/\ggama_2$ with $\Gamma^1/\Gamma^2$, such that the following digram commutes.

\iffalse
Let $\Gamma^1$ be a finite nilpotent $p$-group of nilpotency class $2$, and let $\Gamma^2\subseteq \Gamma^1$ be a central subgroup. Suppose $\mfr{f}:=\Gamma^1/\Gamma^2$ is abelian and is endowed with the structure of a vector space over $\dbF_q$. Assume further that $\ggama_1\supseteq\ggama_2$ are finite Lie-rings, with $\ggama_2$ abelian and $[\ggama_1,\ggama_1]\subseteq \ggama_2$ and $[\ggama_1,\ggama_2]=0$. Assume $\Gamma^1$ acts on $\ggama_1$ and there exists a  $\Gamma^1$-equivariant bijection $\eexp:\ggama_1\to\Gamma^1$, whose inverse we denote by $\llog$, mapping $\Gamma^2$ onto $\ggama_2$ and satisfying the equality
\begin{equation}\label{equaion:heisBCH}\llog(\eexp(X)\eexp(Y))=X+Y+\frac{1}{2}[X,Y],\quad\forall X,Y\in \ggama_1.\end{equation}
In particular, the restriction of $\eexp$ to $\ggama_2$ is an isomorphism of abelian groups. 

\begin{rem}Note that the equalities $\llog(1)=0$ and $\llog((a,b))=[\llog(a),\llog(b)]$ ($a,b\in \Gamma^1$) follow from \eqref{equation:heisBCH} and from the fact that $[\ggama_1,\ggama_1]\subseteq\ggama_2$.
\end{rem}

Finally, we assume that $\ggama_1/\ggama_2\cong \mfr{f}$ and that the diagram below commutes.\fi
\[\begin{tikzpicture}[description/.style={fill=white,inner sep=2pt}]
\matrix (m) [matrix of math nodes, row sep=2.8em,
column sep=3em, text height=1.5ex, text depth=0.25ex]
{\Gamma^1&\ggama_1\\
&\mfr{f}.\\
}; 
\path[->,font=\scriptsize]
(m-1-1) edge[thin] node[above] {$\llog$} (m-1-2);
	
\path[->>,font=\scriptsize]

(m-1-1) edge[thin] node[auto] {} (m-2-2)
(m-1-2) edge[thin] node[auto] {} (m-2-2)
;
\end{tikzpicture}\]
\iffalse
These conditions are summed by the following Diagram (Figure~\ref{figure:basicheisen}), in which all solid lines are group homomorphism.
\begin{figure}[h]\centering
\begin{tikzpicture}[description/.style={fill=white,inner sep=2pt}]
\matrix (m) [matrix of math nodes, row sep=1.8em,
column sep=1.8em, text height=1.5ex, text depth=0.25ex]
{\Gamma^1&\ggama_1&\mfr{f}\\
\Gamma^2&\ggama_2&\set{0}\\
}; 
\path[-,font=\scriptsize]
(m-1-1) edge[thin] node[left] {\rotatebox{90}{$\subseteq$}} (m-2-1)
(m-1-3) edge[thin] node[left] {} (m-2-3)
(m-1-2) edge[thin] node[left] {\rotatebox{90}{$\subseteq$}} (m-2-2);
\path[->,font=\scriptsize]
(m-1-1) edge[dashed] node[below] {$\llog$} node[above] {$1-1$} (m-1-2)
(m-2-1) edge[thin] node[below] {$\llog$} node[above] {$\sim$} (m-2-2)
;

\path[->>,font=\scriptsize]

(m-1-2) edge[thin] node[auto] {} (m-1-3)
(m-2-2) edge[thin] node[auto] {} (m-2-3)
(m-1-1) edge[bend left=55,thin] (m-1-3)
(m-2-1) edge[bend right=55,thin] (m-2-3);

\end{tikzpicture}\caption{Basic Set-up for Applying Heisenberg Lifts.}\label{figure:basicheisen}\end{figure}

\begin{rem}In our application, the groups $\Gamma^2$ will be taken to be of the form $\GG^{k+1}_{2k+1}$, with $k\in\dbN$, and $\Gamma^1$ will be of the form $\GG^{k}_{2k+1}$. In accordance with Lemmas~\ref{lem:exp} and \ref{lem:BCH}, the corresponding Lie-rings would be $\ggama_i:=\g_{2k+1}^{k+(i-1)}$, for $i=1,2$. In all cases discussed, as we will see, the finite vector-space $\mfr{f}$ is identified with $\Gamma^1/\Gamma^2\equiv\dbF_{q^\ell}$.
\end{rem}
\fi

Let $\vartheta$ be an irreducible character of $\Gamma^2$. We describe the set of characters of $\Gamma^1$ which lie above $\vartheta$. As $\Gamma^2\cong\ggama_2$, we have that $\vartheta=\llog^*(\theta)$ for some $\theta\in\left(\ggama_2\right)^\dual$. Define a map on $\Gamma^1\times \Gamma^1$ by \begin{equation}\label{equa:Bthetadefi}B_\theta(a,b):=
\llog^*(\theta)\left((a,b)\right),\quad (a,b\in \Gamma^1),\end{equation}
where $(a,b)=aba^{-1}b^{-1}$ is the group commutator.

As the derived subgroup of $\Gamma^1$ is contained in $\Gamma^2$ and $\Gamma^2$ is abelian, this map is well-defined and corresponds uniquely to an antisymmetric bilinear form 
\[\beta_{\theta}:\mfr{f}\times\mfr{f}\to\dbF_q,\]
such that the diagram \eqref{equa:Pontryadgincomm} commutes, 
where $\psi\in\left(\dbF_q\right)^\dual$ is some fixed non trivial-character of $\dbF_q$.
\begin{equation}\label{equa:Pontryadgincomm}\begin{tikzpicture}[description/.style={fill=white,inner sep=2pt}]
\matrix (m) [matrix of math nodes, row sep=2em,
column sep=.7em, text height=1ex, text depth=0.25ex]
{\Gamma^1/\Gamma^2&\times&\Gamma^1/\Gamma^2&~&\dbC^\times\\
\mfr{f}&\times&\mfr{f}&~&\dbF_q\\
}; 
\path[->,font=\scriptsize]
(m-1-3) edge[thin] node[auto] {$B_\theta$} (m-1-5)
(m-2-3) edge[thin] node[above] {$\beta_\theta$} (m-2-5)
(m-2-5) edge[thin] node[right] {$\psi$} (m-1-5);

\path[-,font=\scriptsize]
(m-1-1) edge[double equal sign distance, thin] node[right] {\rotatebox{90}{$\sim$}} (m-2-1)
(m-1-3) edge[double equal sign distance, thin] node[right] {\rotatebox{90}{$\sim$}} (m-2-3);

\end{tikzpicture}
\end{equation}

Let $\mfr{r}_\theta:=\set{\mbf{x}\in\mfr{f}\mid\beta_\theta(\mbf{x},\mbf{y})=0,\:\text{for all} \mbf{y}\in\mfr{f}}$ be the radical of $\beta_\theta$. Let $\msf{R}_\theta\subseteq\ggama_1$ be the preimage of $\mfr{r}_\theta$ under the reduction map. As the group $\msf{R}_\theta$ is finite and abelian, the character $\theta$ extends to $\msf{R}_{\theta}$ in $\abs{\msf{R}_\theta:\ggama_2}=\abs{\mfr{r}_\theta}$ many ways. Put $R_\theta=\eexp(\msf{R}_\theta)$. Note that since $R_\theta$ coincides with the preimage of $\mfr{r}_\theta$ under the map $\Gamma^1\project \mfr{f}$, we have that $R_\theta$ is a normal subgroup in $\Gamma^1$. Pick an extension $\theta'\in\left(\mfr{r}_\theta\right)^\dual$ of $\theta$, and define $\vartheta':=\llog^*(\theta').$ By Lemma~\ref{lem:formulasexplog} we have that 
\[\vartheta'(a\cdot b)=\theta'\left(\llog(a\cdot b)\right)=\theta'\left(\llog(a)+\llog(b)+\tfrac{1}{2}[\llog(a),\llog(b)]\right)=\vartheta'(a)\vartheta'(b)\]
for any $a,b\in R_\theta$, as $\theta'$ vanishes on commutators in $\msf{R}_\theta$. Moreover, the character $\vartheta'$ is invariant in $\Gamma^1$.

The map $\beta_\theta$ induces a symplectic form on the space $\mfr{f}_\theta:=\mfr{f}/\mfr{r}_\theta$, which we denote by $\beta_\theta$ as well. Let $\mfr{j}_\theta\subseteq\mfr{f}$ be such that $\mfr{j}_\theta/\mfr{r}_\theta$ is a maximal isotropic subspace of $\mfr{f}_\theta$.% That is to say- for any $\mbf{x},\mbf{y}\in\mfr{j}_\theta/\mfr{r}_\theta$ we have that $\beta_\theta(\mbf{x},\mbf{y})=0$ and $\mfr{j}_\theta$ is maximal with respect to this condition. 
 ~It is known that \[\dim_{\dbF_q}\mfr{j}_\theta/\mfr{r}_\theta=
 \dim_{\dbF_q}\mfr{f}/\mfr{j}_\theta=\tfrac{1}{2}\dim_{\dbF_q}\mfr{f}_\theta.\]

Let $\msf{J}_\theta\subseteq \ggama_1$ be the preimage of $\mfr{j}_\theta$ under the reduction map, and let $J_\theta:=\eexp(\msf{J}_\theta)\subseteq \Gamma^1$.
Pick an extension $\theta''$ of $\theta'$ to $\msf{J}_\theta$. As in the case of $R_\theta$, we have that $J_\theta$ is a normal subgroup of $\Gamma^1$. Furthermore,  by Lemma~\ref{lem:formulasexplog}, as in the case of $\vartheta'$, we know that $\vartheta'':=\llog^*(\theta'')$ is a linear character of $J_\theta$.

The main components of the construction are summarized in diagram in Figure~\ref{figure:Heisenberg2}.

\begin{figure}[H]\centering
\begin{tikzpicture}[description/.style={fill=white,inner sep=2pt}]
\matrix (m) [matrix of math nodes, row sep=1.3em,
column sep=1.8em, text height=1.5ex, text depth=0.25ex]
{\Gamma^1&\ggama_1&\mfr{f}\\
J_\theta&\msf{J}_\theta&\mfr{j}_\theta\\
R_\theta&\msf{R}_\theta&\mfr{r}_\theta\\
\Gamma^2&\ggama_2&\set{0}\\
}; 
\path[-,font=\scriptsize]
(m-1-1) edge[thin] node[auto] {} (m-2-1)
(m-2-1) edge[thin] node[auto] {} (m-3-1)
(m-3-1) edge[thin] node[auto] {} (m-4-1)

(m-1-3) edge[thin] node[auto] {} (m-2-3)
(m-2-3) edge[thin] node[auto] {} (m-3-3)
(m-3-3) edge[thin] node[auto] {} (m-4-3)

(m-1-2) edge[thin] node[auto] {} (m-2-2)
(m-2-2) edge[thin] node[auto] {} (m-3-2)
(m-3-2) edge[thin] node[auto] {} (m-4-2);
\path[->,font=\scriptsize]

(m-1-1) edge[dashed] node[below] {$\llog$} (m-1-2)
(m-2-1) edge[dashed] node[below] {$\llog$} (m-2-2)
(m-3-1) edge[dashed] node[below] {$\llog$} (m-3-2)
(m-4-1) edge[thin] node[above] {$\sim$} node[below] {$\llog$} (m-4-2);
\path[->>,font=\scriptsize]
(m-1-2) edge[thin] node[auto] {} (m-1-3)
(m-2-2) edge[thin] node[auto] {} (m-2-3)
(m-3-2) edge[thin] node[auto] {} (m-3-3)
(m-4-2) edge[thin] node[auto] {} (m-4-3)
;
\end{tikzpicture}\caption{Diagram for Heisenberg Lifts}\label{figure:Heisenberg2}\end{figure}
Put $\sigma=\ind_{J_\theta}^{\Gamma^1}(\vartheta'')$. The following holds~\cite[(8.3)]{BushnellFroelich}.
\begin{propo}\label{propo:lift}\begin{enumerate}
\item The character $\sigma$ is irreducible and independent of the choice of $J_\theta$ and $\theta''$.
\item The characters of $\Gamma^1$ lying above $\vartheta$ are in bijection with the set of extensions of $\theta$ to $\msf{R}_\theta$.
\item Any character of $\Gamma^1$ is obtained in this manner.
\end{enumerate}
\end{propo}

\subsubsection*{Computation of Radicals}

The specific case under consideration allows us to obtain a description of the radicals $\mfr{r}_\theta$, which were presented in the previous section, in the case where $\vartheta$ is an irreducible character of $\Gamma^2$ which does not trivialize on $\GG^{2r}_{2r+1}$. Namely, we show the following.
\begin{propo}\label{propo:radicals} Let $r\in\dbN$ be such that $\ell\nmid r$ and let $\vartheta\in\irr(\Gamma^2)$ be non-trivial on $\GG^{2r}_{2r+1}$. Let $\beta_\theta$ be the associated anti-symmetric bilinear form induced from \eqref{equa:Bthetadefi}, and let $\mfr{r}_\theta\subseteq\mfr{f}$ be its radical. The following hold.
\begin{enumerate}
\item If $\ell=2$ then $r_\theta=\set{0}$.
\item Otherwise, if $\ell$ is odd, then $\mfr{r}_\theta$ is a one-dimensional subspace of $\mfr{f}$.
\end{enumerate}
\end{propo}

Let us make some preparations to the proof of Proposition~\ref{propo:radicals}. By Corollary~\ref{corol:dualforabelian}, the character $\vartheta$ is of the form $\llog^*(\phi_y)$ for some $y=y_\vartheta\in\g^{-2r}_{-r+1}.$ Also, the assumption that $\vartheta$ does not vanish on $\GG^{2r}_{2r+1}$ is equivalent to the assumption that $\phi_y$ does not vanish on $\g^{2r}_{2r+1}$ and implies that $\val(y)=-2r$. Pick a lift $Y\in\mfr{g}^{-2r}$ of $y$, and let $Y'\in\OO^\times$ be such that $Y=Y'\nu^{-2r}$. 

Let $x_1,x_2\in \Gamma^1$ be arbitrary and let $X_1,X_2\in\OO$ be such that \[x_i\equiv\eexp( X_i\nu^{r})\pmod{ G^{2r+1}},\quad(i=1,2).\]

Unfolding the definitions of $\phi_y$ and $B_\theta$ (Lemma~\ref{lem:dualpairing} and \eqref{equa:Bthetadefi}), we see that
\begin{equation}\tag{\ref{equa:Bthetadefi}$\:'$}\label{equation:BetaExplicit}B_\theta(x_1,x_2)=\phi_y\left(\log\left( (x_1,x_2)\right)\right)=\Psi\left(\pi^{-1}\cdot\Trd\left(\left(\left[X_1\nu^r,X_2\nu^r\right]\right) \cdot Y'\nu^{-2r}\right)\right).\end{equation}

By taking the character $\psi$ in \eqref{equa:Bthetadefi} to be given by $\lambda\mapsto \Psi(\pi^{-1}\lambda)+\oo$, we obtain an explicit description of the form $\beta_\theta$ by
\begin{equation}
\label{equa:betadef}\beta_\theta(\mbf{x}_1,\mbf{x}_2)=\Tr_{\dbF_{q^\ell}\mid\dbF_q}\left(\left(\tau^2(\mbf{x}_1)\tau(\mbf{x}_2)-\tau(\mbf{x}_1)\tau^2(\mbf{x}_2)\right)\mbf{y}'\right),\quad(\mbf{x}_1,\mbf{x}_2\in\mfr{f}),\end{equation}
where $\mbf{y}'=Y'+\PP\in\dbF_{q^\ell}^\times$ and $\tau\in\Gal(\dbF_{q^\ell}\mid\dbF_q)$ is given by reducing the map $X\mapsto \nu^{-r}X\nu^r$ modulo $\PP$.

\begin{proof}[Proof of Proposition~\ref{propo:radicals}]\begin{enumerate}
\item
Assume $\ell=2$. By assumption, $\tau$ is the unique non-trivial element of $\Gal(\dbF_{q^2}\mid\dbF_q)$, and $\eqref{equa:betadef}$ can be rewritten as
\begin{equation}\tag{\theequation$\:'$}
\label{equa:betadef2}\beta_\theta(\mbf{x}_1,\mbf{x}_2)=\Tr_{\dbF_{q^2}\mid\dbF_q}\left(\left(\mbf{x}_1\tau(\mbf{x}_2)-\tau(\mbf{x}_1)\mbf{x}_2\right)\mbf{y}'\right),\quad(\mbf{x}_1,\mbf{x}_2\in\mfr{f}).\end{equation}
As $\dim_{\dbF_q}\dbF_{q^2}=2$, we have that the space $\mfr{sl}_1(\dbF_{q^2}\mid\dbF_q)$ of elements of trace $0$ is $1$-dimensional. Let $\mbf{i}\in\dbF_{q^2}$ generate this space. By definition we have that $\Tr_{\dbF_{q^2}\mid\dbF_q}(\mbf{i})=\mbf{i}+\tau(\mbf{i})=0$. In particular, $\mbf{i}^2=-\Nr_{\dbF_{q^2}\mid\dbF_q}(\mbf{i})\in\dbF_q^\times$.

Note that by the assumption that $\ell=2$ we have that $\nu^2=\pi$ and hence
%Note that by assumption that $Y$ is traceless, we must have that 
$\Trd(Y')=\Trd(Y\nu^{2r})=\pi^{r}\Trd(Y)=0$. In particular, since the trace and reduction maps commute, we have that $\Tr_{\dbF_{q^2}\mid\dbF_q}(\mbf{y}')=0$, and consequently $\mbf{y'}\in\mbf{i}\dbF_q$. 

Additionally, it holds that ${\Tr_{\dbF_{q^2}\mid\dbF_q}\left(\mbf{x}_1\tau(\mbf{x}_2)\right)=
\Tr_{\dbF_{q^2}\mid\dbF_q}\left(
\tau(\mbf{x}_1)\mbf{x}_2\right)}$ for any $\mbf{x}_1,\mbf{x}_2\in\dbF_{q^2}$, and hence $\mbf{x}_1\tau(\mbf{x}_2)-\tau(\mbf{x}_1)\mbf{x}_2$ is also an element of $\mbf{i}\dbF_q$. In particular, we get that \[\left(\mbf{x}_1\tau(\mbf{x}_2)-\tau(\mbf{x}_1)\mbf{x}_2\right)\mbf{y}'\in\mbf{i}^2\dbF_q=\dbF_q,\]
Applying this into \eqref{equa:betadef2}, we get that $\beta_\theta(\mbf{x}_1,\mbf{x}_2)=2\left(\mbf{x}_1\tau(\mbf{x}_2)-\tau(\mbf{x}_1)\mbf{x}_2\right)\mbf{y}'$. 

One easily verifies that given $\mbf{x}_1\in\dbF_{q^2}$ we may find $\mbf{x}_2\in\dbF_{q^2}$ which satisfies $\frac{\mbf{x}_1}{\tau(\mbf{x}_1)}\ne\frac{\mbf{x}_2}{\tau(\mbf{x}_2)}$, and hence (since $p\ne 2$) $\beta_\theta(\mbf{x}_1,\mbf{x}_2)\ne 0$. Thus $\mfr{r}_\theta=\set{0}$, and the first assertion is proved.
\item 
Assume now that $\ell$ is odd. Note that by the invariance of $\Tr_{\dbF_{q^\ell}\mid\dbF_q}$ under $\Gal(\dbF_{q^\ell}\mid\dbF_q)$ we have that 
\begin{align}\label{equa:radicalequivcond}
\beta_\beta(\mbf{x}_1,\mbf{x}_2)&=\Tr_{\dbF_{q^\ell}\mid\dbF_q}\left(\left(\tau^2(\mbf{x}_1)\tau(\mbf{x}_2)-\tau(\mbf{x}_1)\tau^2(\mbf{x}_2)\right)\mbf{y}'\right)&&\text{(by definition)}\notag\\
&=\Tr_{\dbF_{q^\ell}\mid\dbF_q}\left(\left(\mbf{y}'\tau^2(\mbf{x}_1)-\mbf{x}_1\tau^{-1}(\mbf{y}')\right)\tau(\mbf{x}_2) \right)&&\text{(by invariance under $\tau$)}
\end{align}
for all $\mbf{x}_1,\mbf{x}_2\in\dbF_{q^\ell}$. As the trace pairing is non-degenerate, by \eqref{equa:radicalequivcond},
\begin{align*}\mbf{x}_1\in \mfr{r}_\theta&\iff\beta_\theta(\mbf{x}_1,\mbf{x}_2)=0,\quad\forall \mbf{x}_2\in\dbF_{q^\ell}\\
&\iff \mbf{y}'\tau^2(\mbf{x}_1)=\tau^{-1}(\mbf{y}')\mbf{x}_1,\end{align*}
which occurs if and only if \[\mbf{x}_1=0\quad\text{or}\quad  \frac{\mbf{x}_1}{\tau^2(\mbf{x}_1)}=\frac{\mbf{y}'}{\tau^{-1}(\mbf{y}')}.\]

Since the map $\mbf{x}\mapsto\frac{\mbf{x}}{\tau^2(\mbf{x})}$ is a surjective homomorphism  onto $\SL_1(\dbF_{q^\ell}\mid\dbF_q)$ (e.g. by Hilbert 90) its fibers are of order \[\abs{\dbF_{q^\ell}:\SL_1(\dbF_{q^\ell}\mid\dbF_q)}=q-1.\]

Thus the radical $\mfr{r}_\theta$ of $\beta_\theta$ is a linear subspace of $\dbF_{q^\ell}$ of order $q$ and hence one-dimensional.
%Note that in particular, if $\frac{\tau^2(\mbf{y})}{\tau^3(\mbf{y})}\in\dbF_q\cap\SL_1(\dbF_{q^\ell}\mid\dbF_q)$ then this is equation can be re-written as \[\tau^2(\mbf{x})=\lambda\mbf{x},\]
%for some $\lambda$ of order dividing $\ell$ in $\dbF_q$. If $\lambda\ne 1$ then it is an eigenvalue of $\tau^2$. Since $\lambda^j,\:j=1,\ldots,\ell$ are $\ell$-distinct eigenvalues of $\tau^2$ (with $\mbf{x}^j$ the corresponding eigenvectors), the eigenspaces of $\tau^2$ are $1$-dimensional, as wanted. 

%Otherwise, since $\tau^2$ is non-trivial and $\ell$ is prime, it generates $\Gal(\dbF_{q^\ell}\mid\dbF_q)$, and hence $\mbf{x}$ is in the fixed field of $\Gal(\dbF_{q^\ell}\mid\dbF_q)$, i.e. $\mbf{x}\in\dbF_q$, as wanted.
\end{enumerate}
\end{proof}
\section{Construction of Irreducible Characters}\label{section:generalcase}

In this section we construct the irreducible characters of $\SL_1(D)$. Our construction is an application of the method used by Krakovski, Onn and Singla in \cite{UriRoiPooja} for the study of the regular irreducible representations of $\GL_d(\oo)$ and $\GU_d(\oo)$, for any $d\in\dbN$. The case of $\SL_1(D)$ is especially well-disposed to this type of construction as all irreducible characters of $\SL_1(D)$ are regular (see~\cite{Hill} for the definition of a regular character). The prototype for this type of construction is the computation of the representation zeta function of $\SL_2(\mfr{o})$, where $\mfr{o}$ is a compact discrete valuation ring, which was completed by Jaikin-Zapirain in \cite[\S~7]{jaikin}. Note that all irreducible characters of $\SL_2(\mfr{o})$ are regular as well.

%In Subsection~\ref{subsection:Liecorres} we will will describe an explicit bijection between subquotients of $G$, which are finite nilpotent groups of nilpotency class smaller or equal to $2$, and certain finite Lie-rings. This would set up a correspondence between the set of irreducible representations of these finite groups and the Pontryagin dual of the associated Lie-ring, which we refer to as the\textit{ Lie-correspondence for subquotients of $\SL_1(D)$}. In Subsection we will describe the method of \textit{Heisenbreg Lifts}, which allows us to study the representation theory of finite nilpotent $p$-groups of nilpotency class $2$, under certain assumptions on the 'Lie-structure' of the group. Finally, in Subsection~\ref{subsect:construction} we will complete the construction of all irreducible representations of $\SL_1(D)$ and compute the dimensions of all such representations, as well as prove Theorem~\ref{theo:monomial}.In this Subsection we complete the construction of all irreducible representations of $G$. In general lines, the construction is performed as follows. 

Recall that the level of a character $\varphi$ of $G$ is the minimal number $m$ such that $\varphi$ is trivial on $G^{m+1}$. We consider a fixed irreducible character $\varphi$ of $G$ of level $m\in\dbN_0$.

Let $0\le r<m+1$ and let $\vartheta$ be an irreducible constituent of the restriction of $\varphi$ to $G^r$. Recall that the inertia subgroup of the character $\vartheta$ in $G$ is defined by 
\begin{equation}\label{equa:inertiadefi}I_G(\vartheta):=\set{g\in G\mid \vartheta={}^g\vartheta},\end{equation}
where ${}^g\vartheta(x):=\vartheta(g^{-1}xg)$ for all $x\in G^r$. Also recall that in the case where $\vartheta$ extends to a character $\hat{\vartheta}$ of $I_G(\vartheta)$, the induced character $\ind_{I_G(\vartheta)}^G(\hat{\vartheta})$ is irreducible \cite[Theorem~6.11]{Isaacs}.

In the case at hand, Lemmas~\ref{lem:exp} and \ref{lem:dualpairing} allow us to identify $\vartheta$ with an element of a suitable subquotient of the Lie-ring $\mfr{sl}_1(D)$ in a $G$-equivariant manner. In particular, the inertia subgroup of $\vartheta$ in $G$ is equal to the $m'$-th congruence stabilizer of an element of $D$, for a suitable $m'\in\dbZ$. In the case where $m$ is odd, or is divisible by $\ell$, the decomposition shown in Proposition~\ref{propo:centralizerformula} allows us to extend $\vartheta$ to a character of $I_G(\vartheta)$, whose induction is equivalent to $\varphi$. 

In the general case, we show the following theorem, which implies Theorem~\ref{theo:monomial}.

 %Recall that the set of representations of $G$ of level $m$ can be identified with the set of representations of $\GG_{m+1}$ which are non-trivial on $\GG^m_{m+1}$. In particular, $\vartheta$ can be identified with an irreducible representation of the finite abelian group $\GG^{\ceil{m/2}}_{m+1}$ and hence is $1$-dimensional. Additionally, Lemmas~\ref{lem:exp} and~\ref{lem:dualpairing} allow us to identify $\vartheta$ with an element of a 
\begin{theo}\label{theo:proveseverything}Let $\varphi\in\irr(G)$ have level $m\in\dbN_0$. Then
\begin{enumerate}
\item The degree of $\varphi$ is \[d_m^\ell(q)=\begin{cases}q^{\frac{\ell-1}{2}m}&\text{if }\ell\mid m,\\
\frac{q^\ell}{\iota\cdot(q-1)}q^{\frac{\ell-1}{2}(m-1)}&\text{otherwise}.
\end{cases}\]
\item Assuming $\ell>2$ or $\ell=2$ and $m\not\equiv_{(4)}2$, there exists an open subgroup $J\subseteq I_G(\vartheta)$ and a linear character $\hat{\vartheta}\in\irr(J)$ such that $\varphi\simeq \ind_{J}^G(\hat{\vartheta}).$
\end{enumerate}
\end{theo}

The proof of Theorem~\ref{theo:proveseverything} is divided into cases, according to $m$, by ascending level of complexity, given in Sections~\ref{subsect:oddlevel},~\ref{subsect:divby2ell} and~\ref{subsect:evenlevel}.

Note that, as the group $\GG_1\cong\SL_1(\dbF_{q^\ell}\mid\dbF_q)$ is abelian, Theorem~\ref{theo:proveseverything} holds trivially in the case $m=0$.

We will require the following lemma, which can be proved by a simple application of $\nu$-expansions.

 \begin{lem}\label{lem:mctraceless}Let $y\in D$ have $\Trd(y)=0$. Then $\mc(y)=\ell\cdot\val(y).$
\end{lem}
%\begin{proof}
%That $\mc(y)\ge \ell\val(y)$ is immediate from the definition. Suppose $\mc(y)>\ell\val(y)$, then in particular there exists $\lambda\in K$ such that $\val(y-\lambda)>val(y)$, and hence $\val(\lambda)=\val(y-(y-\lambda))=\min\set{\val(y),\val(y-\lambda)}=\val(y).$ Hence $\val(y)\in \dbZ$, and $y\equiv\lambda(\mod\pp^{\val(y)+1}).$ But then $\Trd(y)\equiv \ell\lambda(\mod\pp^{\val(y)+1})$ and in particular, is non-zero. A contradiction.
%\end{proof}

\subsection{Characters of Odd Level}\label{subsect:oddlevel}
Assume $m$ is odd, and put  $r=\frac{m+1}{2}$. By definition of the level of $\varphi$, it may be identified with a character of the finite group $\GG_{m+1}$ which is non-trivial on $\GG_{m+1}^m$. Let $\vartheta$ be an irreducible constituent of the restriction of $\varphi$ to $\GG^{r}_{m+1}$. By Corollary~\ref{corol:dualforabelian}, there exists $y=y_\vartheta\in\g^{-m}_{-r+1}$ such that $\vartheta=\llog^*(\phi_y)$. Let $Y\in\mfr{g}^{-m}$ be a lift of $y$. As the bijection of Corollary~\ref{corol:dualforabelian} is $G$-equivariant, we have that
\[I_G(\vartheta)=\stab_G(y)=\St^{-r+1}_G(Y).\]
Here $\stab_G(y)$ denotes the stabilizer of $y$ under the conjugation action of $G$. 
\iffalse
Assume $m$ is odd. %, and write $m=2k\ell+\lambda$ with $0<\lambda<2\ell$ odd. 
By definition of level, $\varphi$ may be considered a character of the finite group $\GG_{m+1}$, which is non-trivial on $\GG_{m}$. Let $\vartheta$ be an irreducible constituent of the restriction of $\varphi$ to $\GG^{r}_{m+1}$. By Lemma~\ref{lem:exp}, $\GG^{k\ell+\frac{m+1}{2}}_{m+1}$ is abelian and $\vartheta$ is of the form $\log^*(\theta)$ for some $\theta\in\left(\g^{r}_{m+1}\right)^\dual$.

Let $0<\lambda<2\ell$ be such that $m=2k\ell+\lambda$. Applying Proposition~\ref{propo:dual} (with $f=\lambda+1$) we have that $\theta=\phi_y$, for some $y\in\g^{-\lambda}_{k\ell-\frac{\lambda-1}{2}}$, where $\phi_y(x)$ as in Lemma~\ref{lem:dualpairing}. Let $Y\in \mfr{g}^{-\lambda}$ be some lift of $y$. 

As all isomorphisms described above are $G$-equivariant, and by the definition of the congruence stabilizer, we have that
\[I_G(\vartheta)=\stab_G(\theta)=\CC_G(y)=\St^{k\ell-\frac{\lambda-1}{2}}_G(Y).\]\fi
Furthermore, by Lemma~\ref{lem:mctraceless} and since $\phi_y$ does not trivialize on $\g_{m+1}^{m}$, we have that $\mc(Y)=\ell\cdot\val(Y)=-m$. By Proposition~\ref{propo:centralizerformula}, we deduce that \[I_G(\vartheta)=\CC_G(Y)\cdot G^{r}.\]

Reducing modulo $G^{m+1}$, we describe our situation in Figure~\ref{figure:oddlevel}.    
\begin{figure}\centering
\begin{tikzpicture}[description/.style={fill=white,inner sep=2pt}]
\matrix (m) [matrix of math nodes, row sep=1em,
column sep=1.5em, text height=1.5ex, text depth=0.25ex]
{&I_{\GG_{m+1}}(\vartheta)\\
\msf{C}_{m+1}(Y)&&\GG^{r}_{m+1}\\
& \msf{C}_{m+1}(Y)\cap \GG^{r}_{m+1}& \\~\\
&\set{1}\\ };
%\draw[double,double distance=5pt] (m-1-1) – (m-1-5);
\path[-,font=\scriptsize]
(m-1-2) edge[thin] node[auto] {} (m-2-1)
	    edge[thin] node[auto] {} (m-2-3)
%(m-2-4) edge[double equal sign distance, thick] node[right] {} (m-2-3)
(m-3-2) edge[thin] node[auto] {} (m-5-2)
		edge[thin] node[auto] {} (m-2-1);
\path[-,shorten >=.4cm,font=\scriptsize]
(m-2-3)	edge[thin] node[auto] {} (m-3-2)
;
\end{tikzpicture}\caption{Diagram for odd level}\label{figure:oddlevel}\end{figure}
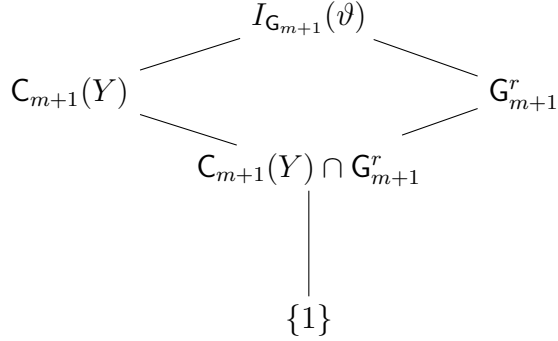
Here $\msf{C}_{m+1}(Y)$ denotes the reduction of $\CC_G(Y)$ modulo $G^{m+1}$. 

As the group $\msf{C}_{m+1}(Y)$ is abelian, we can select a linear character $\chi$ of $\msf{C}_{m+1}(Y)$, lying under $\varphi$ and extending the restriction of $\vartheta$ to $\msf{C}_{m+1}(Y)\cap \GG_{m+1}^{r}$. In this situation, the characters $\chi$ and $\vartheta$ glue to a linear character $\hat{\vartheta}$ of $I_{\GG_{m+1}}(\vartheta)$, extending $\vartheta$ to $I_{\GG_{m+1}}(\vartheta)$ and lying under $\varphi$.

By \cite[Theorem 6.11]{Isaacs}, the induced character $\ind_{I_G(\vartheta)}^G(\hat{\vartheta})$ is irreducible and equal to $\varphi$. Furthermore, by a simple computation based on Propositions~\ref{propo:indexcent} and~\ref{propo:centrQtoG}, we have that 
\begin{align*}\varphi(1)=&\abs{\GG_{m+1}:I_{\GG_{m+1}}(\vartheta)}=\abs{G:\St^{-r+1}_G(Y)}\\=&\begin{cases}q^{\frac{\ell-1}{2}\cdot m}&\text{if }\ell\mid m,\\
\frac{q^\ell-1}{\iota\cdot(q-1)}\cdot q^{\frac{\ell-1}{2}\cdot (m-1)}&\text{if } \ell\nmid m.
\end{cases}\end{align*}
Thus, Theorem~\ref{theo:proveseverything} is true for characters of odd level.

\subsection{Characters of Level Divisible by $2\ell$} \label{subsect:divby2ell}
Assume now that $m$ is divisible by $2\ell$ and put $r=\frac{m}{2}$. This case is somewhat similar to the case of odd level, as the fact that $\GG_{m+1}^r$ is abelian group remains true in this case (see~Section~\ref{section:chatcorres}). A slight modification is required to account for the fact that Corollary~\ref{corol:dualforabelian} does not hold for the group $\GG^{\frac{m}{2}}_{2m+1}$.

Again, $\varphi$ may be regarded as a character of $\GG_{m+1}$, which is non-trivial on $\GG^m_{m+1}$. Let $\vartheta'$ be an irreducible constituent of the restriction of $\varphi$ to $\GG^{r}_{m+1}$. Then $\vartheta'$ is a linear character.
%However, a simple argument as in the previous case is unavailable, as Lemma~\ref{lem:exp} does not imply that $\GG^{k\ell}_{2k\ell+1}$ is isomorphic $\g^{2k\ell}_{2k\ell+1}$ in this case. Nonetheless, this case is not much more complicated.

Let $\vartheta$ be the restriction of $\vartheta'$ to ${\GG^{r+1}_{m+1}}$. Corollary~\ref{corol:dualforabelian} is applicable to $\GG^{r+1}_{m+1}$ and hence

\begin{propo}\label{propo:dual} Let $r,m\in\dbN$ be given with $r\le m$. Then, the Pontryagin dual of $\g^r_m$ is isomorphic to the group $\g^{-m+1}_{-r+1}$ via a $G$-equivariant isomorphism, which intertwines the coadjoint action of $G$ on $\left(\g^r_m\right)^\dual$ and the adjoint action on $\g^{-m+1}_{-r+1}.$
\end{propo} $\vartheta=\llog^*(\phi_y)$ for some $y\in \g^{-m}_{-r}$. Picking a lift $Y\in \mfr{g}^{-m}$ of $y$, as in the previous case, we have that
%By Lemma~\ref{lem:exp}, $\vartheta=\log^*(\theta)$ for some $\theta\in\left(\g^{k\ell+1}_{2k\ell+1}\right)^\dual$, and by Proposition~\ref{propo:dual} (applied with $f=1$) we have that $\theta=\phi_y$, with $y\in\g_{k\ell}$. Pick $Y\in\mfr{g}$ to be a lift of $y$. Then, as in the previous case, we have that 
\[I_G(\vartheta)=\stab_G(y)=\St^{-r}_G(Y).\]

Furthermore, by Lemma~\ref{lem:mctraceless} and the fact that $\phi_y$ does not trivialize on $\g^{m}_{m+1}$, we have that $\mc(Y)=-m$. Thus, by Proposition~\ref{propo:centralizerformula}, we have that
\[I_G(\vartheta)=\CC_G(Y)\cdot G^{r}.\] This situation is summarized in Figure~\ref{figure:level2kl}.

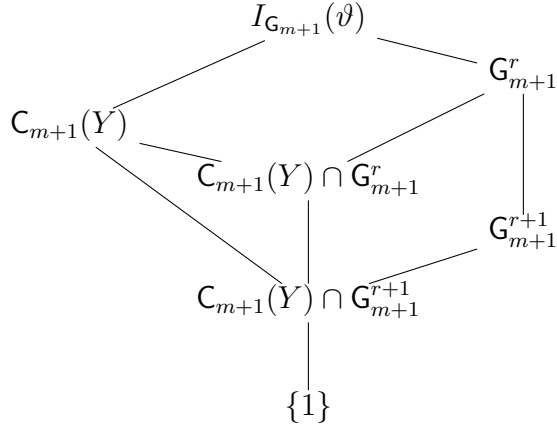
\begin{figure}\centering
\begin{tikzpicture}[description/.style={fill=white,inner sep=2pt}]
\matrix (m) [matrix of math nodes, row sep=.5em,
column sep=1.5em, text height=1ex, text depth=0.25ex]
{&I_{\GG_{m+1}}(\vartheta)\\
&&\GG^{r}_{m+1}\\
\msf{C}_{m+1}(Y) & &  \\
& \msf{C}_{m+1}(Y)\cap \GG^{r}_{m+1}& 
\\&&\GG_{m+1}^{r+1}\\\\
&\msf{C}_{m+1}(Y)\cap \GG_{m+1}^{r+1}& \\~\\
&\set{1}\\ };
%\draw[double,double distance=5pt] (m-1-1) – (m-1-5);
\path[-,font=\scriptsize]
(m-1-2) edge[thin] node[auto] {} (m-3-1)
(m-1-2) edge[thin] node[auto] {} (m-2-3)
(m-3-1) edge[thin] node[auto] {} (m-4-2)
(m-2-3) edge[thin] node[auto] {} (m-4-2)
(m-3-1) edge[thin] node[auto] {} (m-7-2)
(m-5-3) edge[thin] node[auto] {} (m-7-2)
(m-4-2) edge[thin] node[auto] {} (m-7-2)
(m-2-3) edge[thin] node[auto] {} (m-5-3)
(m-9-2) edge[thin] node[auto] {} (m-7-2);
\end{tikzpicture}\caption{Diagram for  $2k\ell$}\label{figure:level2kl}\end{figure}

As $\msf{C}_{m+1}(Y)$ is abelian, we may pick a linear character $\chi$ of $\msf{C}_{m+1}(Y)$ which lies under $\varphi$ and extends the restriction of $\vartheta'$ to $\msf{C}_{m+1}(Y)\cap \GG_{m+1}^{r}$. Furthermore, we have that $\vartheta'$ extends $\vartheta$ to $\GG_{m+1}^{r}$. Thus, the characters $\chi$ and $\vartheta'$ glue to a linear character $\hat{\vartheta}$ of $I_{\GG_{m+1}}(\vartheta)$, which extends $\vartheta$ and lies under $\varphi$. As in the previous case, the induced character $\ind_{I_G(\vartheta)}^G(\hat{\vartheta})$ is irreducible and equal to $\varphi$. Additionally, by Propositions~\ref{propo:indexcent} and \ref{propo:centrQtoG}
\[\varphi(1)=\abs{G:I_G(\vartheta)}=\abs{G:\St^{-r}_G(Y)}=q^{\frac{\ell-1}{2}\cdot m}=d^\ell_{m}(q),\]
as wanted.

\subsection{Characters of Even Level}\label{subsect:evenlevel}

Lastly, we consider the case where the level of $\varphi$ is an even number not divisible by $2\ell$. To complete the construction and prove Theorem~\ref{theo:proveseverything} in this case, we will use the tools described in Section~\ref{subsect:heisenberg}. As seen in Proposition~\ref{propo:radicals}, the case where $\ell=2$ is somewhat different from the case where $\ell$ is odd, and is treated separately.

Let us recall the notation of Section~\ref{subsect:heisenberg}. We assume that $\varphi$ has even level $m=2r$ and that $r$ is not divisible by $\ell$. We write
\[\Gamma^i=\GG^{r+(i-1)}_{2r+1}\quad\text\quad
\ggama_i=\g^{r+(i-1)}_{2r+1},\quad(i=1,2),\]
and $\mfr{f}=\Gamma^1/\Gamma^2\cong\dbF_{q^\ell}$.

%We chose an irreducible constituent $\vartheta$ of the restriction of $\varphi$ to $\Gamma^2$. Then $\vartheta=\llog^*(\theta)$ for some $\theta\in(\ggama^2)^\dual$, which is of the form $\phi_y$ for some $y\in\g^{-2r}_{-r+1}$. We fix $Y\in\mfr{g}^{-2r}$ a lift of $y$. We have that $\mc(Y)=\ell\val(Y)=-2r$.

\subsubsection*{Division Algebras of Odd Degree}
Let $\sigma\in\irr(\Gamma^1)$ be a constituent of the restriction of $\varphi$ to $\Gamma^1$ and let $\vartheta\in \irr(\Gamma^2)$ lie under $\sigma$. Put $\theta=\eexp^*(\vartheta)\in(\ggama_2)^\dual$. By Proposition~\ref{propo:lift}(2), the character $\sigma$ corresponds to a unique extension $\theta'\in (\msf{R}_\theta)^\dual$ lying above $\theta$. By extending $\theta'$ to $\ggama_1$ and applying Proposition~\ref{propo:dual}, there exists an element $y'\in\g^{-2r}_{-r+1}$ such that $\theta'$ coincides with the restriction of $\phi_{y'}$ to $\msf{R}_\theta$. 

Let $Y\in\mfr{g}^{-2r}$ be a lift of $y'$ and put $y=Y+\mfr{g}^{-r}\in\g^{-2r}_{-r}$. One checks from the definition of $\phi_y$ that $\phi_y=\theta$. As before, we have that \[I_G(\vartheta)=\St^{-r}_G(Y)=\CC_G(Y)\cdot G^{r}.\]

Moreover, the bijection of Proposition~\ref{propo:lift}(1) implies that
\[I_G(\vartheta')\subseteq I_G(\vartheta)=\CC_G(Y)\cdot G^r\subseteq I_G(\sigma)=I_G(\vartheta'),\]
where $\vartheta'=\llog^*(\theta')\in\irr(R_\theta)$.

The main components of the construction appear in Figure~\ref{figure:2ml+evenlambda}. The groups $\mfr{j}_\theta,\:\msf{J}_\theta, J_\theta$ will arise by picking an isotropic subspace of $\mfr{f}_\theta=\mfr{f}/\mfr{r}_\theta$. The group $\msf{C}_{m+1}(Y)$ denotes the image of $\CC_G(Y)$ in $\GG_{m+1}$. %Note that by Lemma~\ref{lem:1dimrad} the group $\mfr{r}_\theta=\msf{R}_\theta/\Gamma^2$ is either trivial (if $\ell=2$) or of order $q$ (if $\ell>2$).
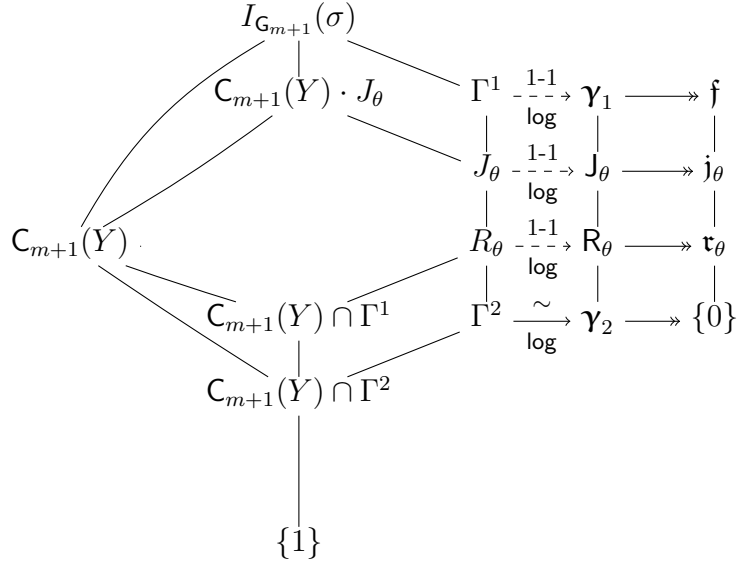
\begin{figure}[h!]
\begin{tikzpicture}[description/.style={fill=white,inner sep=2pt}]
\matrix (m) [matrix of math nodes, row sep=1.2em,
column sep=1.8em, text height=1ex, text depth=0.25ex]
{%&\GG_{m+1}&&\g_{m+1}\\~\\
&I_{\GG_{m+1}}(\sigma)\\
&\msf{C}_{m+1}(Y)\cdot J_\theta&\Gamma^1&\ggama_1&\mfr{f}\\
&&J_\theta&\msf{J}_\theta&\mfr{j}_\theta\\
\msf{C}_{m+1}(Y)&&R_\theta&\msf{R}_\theta&\mfr{r}_\theta\\
&\msf{C}_{m+1}(Y)\cap \Gamma^1&\Gamma^2&\ggama_2&\set{0}\\
&\msf{C}_{m+1}(Y)\cap \Gamma^2\\~\\
&\set{1}&&&\\}; 

\path[-,font=\scriptsize]

%(m-1-2) edge[thin] node[auto] {} (m-3-2)
%(m-1-4) edge[thin] node[auto] {} (m-4-4)
(m-1-2) edge[thin] node[auto] {} (m-2-2)
		edge[bend right=15, thin] node[auto] {} (m-4-1)
		edge[thin] node[auto] {} (m-2-3)
(m-3-3) edge[thin] node[auto] {} (m-2-3)
		edge[thin] node[auto] {} (m-2-2)
		edge[thin] node[auto] {} (m-4-3)
(m-4-1) edge[bend right=3, thin] node[auto] {} (m-2-2)
		edge[thin] node[auto] {} (m-4-1)
		edge[bend right=1,thin] node[auto] {} (m-6-2)
(m-4-1)	edge[bend right=1,thin] node[auto] {} (m-5-2)
%(m-7-3) edge[thin] node[auto] {} (m-8-2)
(m-4-3) edge[thin] node[auto] {} (m-5-2)
(m-5-3) edge[thin] node[auto] {} (m-4-3)
(m-6-2) edge[thin] node[auto] {} (m-5-3)
		edge[thin] node[auto] {} (m-5-2)
		edge[thin] node[auto] {} (m-8-2)
		
(m-3-4) edge[thin] node[auto] {} (m-2-4)
		edge[thin] node[auto] {} (m-4-4)
(m-5-4) edge[thin] node[auto] {} (m-4-4)	

(m-3-5) edge[thin] node[auto] {} (m-2-5)
		edge[thin] node[auto] {} (m-4-5)
(m-5-5) edge[thin] node[auto] {} (m-4-5)		
;

\path[->,font=\scriptsize]
(m-2-3) edge[dashed] node[below] {$\llog$} node[above] {1-1} (m-2-4)
(m-3-3) edge[dashed] node[below] {$\llog$} node[above] {1-1} (m-3-4)
(m-4-3)	edge[dashed] node[below] {$\llog$} node[above] {1-1} (m-4-4)
(m-5-3) edge[thin] node[below] {$\llog$} node[above] {$\sim$} (m-5-4)	
;
\path[->>,font=\scriptsize]
(m-2-4) edge[thin] node[auto] {} (m-2-5)
(m-3-4) edge[thin] node[auto] {} (m-3-5)
(m-4-4) edge[thin] node[auto] {} (m-4-5)
(m-5-4) edge[thin] node[auto] {} (m-5-5)
;
\end{tikzpicture}\caption{Diagram for even level, coprime with $2\ell$, $\ell>2$}\label{figure:2ml+evenlambda}\end{figure}

%Let us analyse the action of the group $\CC_G(Y)$ on $\mfr{f}$. Let $g\in {\OO^\times}$ be arbitrary, and let $X=\pi^k X'\nu^\eta\in \mfr{g}^{k\ell+\eta}$ be given, with $X'\in\OO$. Let $\mbf{g}\in\dbF_{q^\ell}^\times$ and $\mbf{x}'\in\dbF_{q^\ell}$ denote the image of $g$ and $X'$ modulo $\PP$. Then \[g\cdot X\cdot g^{-1}=\pi^k \cdot (g X'\nu^{\eta}g^{-\eta}\nu^{-\eta})\cdot \nu^{\eta}.\]
%Taking $\tau\in\Gal(\dbF_{q^\ell}\mid\dbF_q)$ be as above, we get that the action of ${\OO^\times}$ on $\mfr{f}$ is given by
%\begin{equation}\label{equa:conjCGYonf}\mbf{g}.\mbf{x}':=\mbf{g}\cdot \tau^{-1}(\mbf{g}^{-1})\cdot \mbf{x}.\end{equation}

%In particular, it follows from \eqref{equa:conjCGYonf}, that in the case where $\ell>2$ the action of $\CC_G(Y)$ on $\mfr{%f}$ is trivial. Indeed, as $\ell>2$ and $\val(Y)=\frac{-2\eta}{\ell}\notin\dbZ$, the extension  $K(y)/K$ is totally ramified, and hence any element of $\CC_G(Y)$ is congruent modulo $\PP$ to an element of $K$. In particular, any choice of $\mfr{j}_\theta\subseteq\mfr{f}$ such that $\mfr{j}_\theta/\mfr{r}_\theta$ is maximal isotropic is fixed by $\CC_G(Y).$

Note that the assumption that $\ell\nmid r$ implies that ${\val(Y)=-\frac{2r}{\ell}\notin \dbZ}$ and in particular that $Y$ is ramified. Consequently, the group $\CC_G(Y)$ is the intersection of $G$ with the ring of integers of a totally ramified extension $K(Y)/K$ and hence the image of $\CC_Y(G)$ modulo $1+\PP$ is contained in $\dbF_q^\times$. In particular, it follows that the conjugation action of $\CC_G(Y)$ on $\Gamma^1$ reduces to the trivial action of $\CC_G(Y)$ on $\mfr{f}$.

Let $\mfr{j}_\theta\subseteq \mfr{f}$ be such that $\mfr{j}_\theta/\mfr{r}_\theta$ is a maximal isotropic subspace. Let $J_\theta$ be the preimage of $\mfr{j}_\theta$ under the projection $\Gamma^1\to	\mfr{f}$. Since $\CC_G(Y)$ fixes $\mfr{j}_\theta$, it normalizes $J_\theta$. Let $\vartheta''$ be an extension of $\vartheta'$ to $J_\theta$. As $\msf{C}_{m+1}(Y)$ is abelian and normalizes $J_\theta$, by choosing a constituent $\chi\in\irr(\msf{C}_{m+1}(Y))$ lying under $\varphi$ as in the previous cases, the characters $\chi$ and $\vartheta''$ glue to a linear character $\hat{\vartheta}$ of $\msf{C}_{m+1}(Y)J_\theta$. Put $\Theta:=\ind_{\msf{C}_{m+1}(Y)J_\theta}^{I_{\GG_{m+1}(\sigma)}}(\hat{\vartheta})$.
\begin{samepage}
\begin{propo}\label{propo:Theta}\begin{enumerate}
\item The character $\Theta$ is irreducible and of degree $q^{\frac{\ell-1}{2}}$.
\item The inertia subgroup of $\Theta$ in $\GG_{m+1}$, defined as in \eqref{equa:inertiadefi}, equals $I_{\GG_{m+1}}(\sigma)$.
\end{enumerate}
\end{propo}
\end{samepage}
\begin{proof}
From the explicit definition \eqref{equation:BetaExplicit} of $B_\theta$, we have that ${{\msf{C}_{m+1}(Y)\cap \Gamma^1} \subseteq R_\theta}$. It follows that $(\msf{C}_{m+1}(Y)\cdot J_\theta)\cap \Gamma^1=J_\theta$, as \[J_\theta\subseteq (\msf{C}_{m+1}(Y)\cdot J_\theta)\cap \Gamma^1\subseteq J_\theta\cdot(\msf{C}_{m+1}(Y)\cap\Gamma^1)\subseteq J_\theta,\]
Hence \[\Theta(1)=\abs{I_{\GG_{m+1}}(\sigma):\msf{C}_{m+1}(Y)J_\theta}=\abs{\Gamma^1:J_\theta}=q^{\frac{\ell-1}{2}},\] and the claim regarding the degree of $\Theta$ is proved.

%By Frobenius Reciprocity, there exists a non-zero map of representations of $\msf{C}_{m+1}(Y)\cdot J_\theta$ from $\hat{\vartheta}$ into $\res^{I_{\GG_{m+1}}(\sigma)}_{\msf{C}_{m+1}J_{\theta}}(\Theta)$. Restricting this map to a map of representations of $J_\theta$, we obtain an embedding of $\vartheta''$ as a sub-representation of $\res^{I_{\GG_{m+1}(\sigma)}}_{J_\theta}(\Theta)$.
By Frobenius Reciprocity, we have that  \[\left[\hat{\vartheta},\res^{I_{\GG_{m+1}}(\sigma)}_{\msf{C}_{m+1}J_{\theta}}(\Theta)\right]_{\msf{C}_{m+1}(Y)\cdot J_\theta}=\left[\Theta,\Theta\right]_{I_{\GG_{m+1}}(\sigma)}\ne 0.\] 
Restricting the inner-product on the left-hand side to the group $J_\theta$, one easily deduces that $\left[\vartheta'',\res^{I_{\GG_{m+1}}(\sigma)}_{J_{\theta}}(\Theta)\right]_{J_\theta}\ne 0.$
%Restricting this map to a map of representations of $J_\theta$, we obtain an embedding of $\vartheta''$ as a sub-representation of $\res^{I_{\GG_{m+1}(\sigma)}}_{J_\theta}(\Theta)$.
%Furthermore, note that $\vartheta''$ naturally embeds into $\res^{I_{\GG_{m+1}(\sigma)}}_{J_\theta}(\Theta)$, by restricting any map $\hom_{\msf{C}_{m+1}(Y)J_\theta}\left(\hat{\vartheta},\res^{I_{\GG_{,+1}(\sigma)}}_{\msf{C}_{m+1}(Y)J_\theta}(\Theta)\right)$ to a map of representations of $J_\theta$. Thus, 
Again, by Frobenius Reciprocity, we have that 
\[\left[\ind_{J^\theta}^{\Gamma^1}(\vartheta''),\res^{I_{\GG_{m+1}(\sigma)}}_{\Gamma^1}(\Theta)\right]_{\Gamma^1}=\left[\vartheta'',\res^{I_{\GG_{m+1}(\sigma)}}_{J_\theta}(\Theta)\right]_{J_\theta}\ne 0.\]

%\[\hom_{\Gamma^1}\left(\ind_{J^\theta}^{\Gamma^1}(\vartheta''),\res^{I_{\GG_{m+1}(\sigma)}}_{\Gamma^1}(\Theta)\right)\cong \hom_{J_\theta}\left(\vartheta'',\res^{I_{\GG_{m+1}(\sigma)}}_{J_\theta}(\Theta)\right)\ne 0.\]
Since $\ind_{J^\theta}^{\Gamma^1}(\vartheta'')=\sigma$ is irreducible (Proposition~\ref{propo:lift}(1)), and of the same degree as $\Theta$, we deduce that the value on the left-hand side of the above equation is $1$ and hence $\res^{I_{\GG_{m+1}}(\sigma)}_{\Gamma^1}(\Theta)=\sigma.$ It follows that $\Theta$ is irreducible, for if not, its restriction to $\Gamma^1$ would have been reducible as well. Assertion (1) of the proposition is now proved.

To prove (2), note that is $\Theta^g$ and $\Theta$ are equal, for some $g\in \GG_{m+1}$, then in particular $g$ stabilizes the restriction of $\Theta$ to $\Gamma^1$. But this restriction is equal to $\sigma$, whence $g\in I_G(\sigma)$. The converse inclusion is clear.
\end{proof}
\begin{corol}The induced character $\ind_{\msf{C}_{m+1}(Y)J_\theta}^{\GG_{m+1}}(\hat{\vartheta})$ is irreducible and its pullback to $G$ is isomorphic to $\varphi$.\end{corol}
In particular, we have that
\begin{align*}\varphi(1)&=\abs{\GG_{m+1}:I_{\GG_{m+1}}(\sigma)}\cdot \dim(\Theta)=\abs{G:\St_G^{-r}(Y)}\cdot q^{\frac{\ell-1}{2}}\\
&=\frac{q^\ell-1}{\iota\cdot(q-1)}\cdot q^{\frac{\ell-1}{2}(m-1)}%&\text{if }\ell\ne 2,
%\color{red}q^{4k+1}&\text{if }\ell=2,\:\eta=1\color{black}
%\end{cases}\\&=
=d_{m}^\ell(q).
\end{align*}

The construction and proof of Theorem~\ref{theo:proveseverything} is now complete for $\ell>2$.

\subsubsection*{Quaternion Algebras}
The argument applied in the previous case fails in the quaternion case, as in the setting where $\ell=2$ and $r$ is even, the image of $\CC_G(Y)$ modulo $1+\PP$ is the group $\SL_1(\dbF_{q^2}\mid\dbF_q)$. In particular, the conjugation action of $\CC_G(Y)$ on $\Gamma^1$ induces a transitive action on the set of lines in $\mfr{f}=\dbF_{q^2}$. It follows that no maximal isotropic subspace of $\mfr{f}$ is fixed by $\CC_G(Y)$. We supplement the argument of the previous case with a different argument, which proves the degree formula, but does not prove the monomiality of $\varphi$.

As before we pick a constituent $\sigma\in\irr(\Gamma^1)$ of $\varphi$ and $\vartheta\in \irr(\Gamma^2)$ a constituent of $\sigma$. As the radical of the form $\beta_\theta$ is trivial (Proposition~\ref{propo:radicals}(1)), we have that $\sigma=\ind_{J_\theta}^{\Gamma^1}(\vartheta')$, for any $J_\theta$ as in Section~\ref{subsect:heisenberg} and any extension $\vartheta'\in\irr(J_\theta)$ of $\vartheta$. Also, we have that $I_{G}(\sigma)=I_G(\vartheta)$. Fix such a subgroup $J_\theta\subseteq \Gamma^1$ and extension $\vartheta'\in\irr( J_\theta)$.
%{\stackrel{{\tiny{\it def.}}}{=}}

Let $P(Y):=\CC_{G^1}(Y)$ be the Sylow pro-$p$ subgroup of $\CC_G(Y)$. Note that by its definition, %$\CC_G(Y)/P(Y)\cong\SL_1(\dbF_{q^2}\mid\dbF_q)$ is cyclic, and that 
$P(Y)$ acts trivially on $\mfr{f}$. Our diagram in this case appears in Figure~\ref{figure:4k+1}, in which $\msf{C}_{m+1}(Y)$ and $\msf{P}_{m+1}(Y)$ are the images of $\CC_G(Y)$ and $P(Y)$ in $\GG_{m+1}$, and $\mfr{j}_\theta$ is a maximal isotropic subspace (i.e. a line) in $\mfr{f}$.

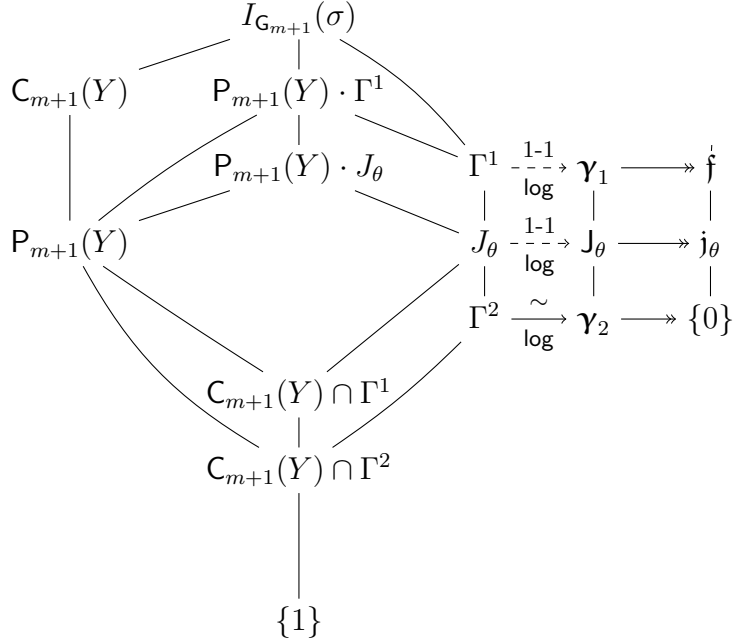
\begin{figure}[h]
\begin{tikzpicture}[description/.style={fill=white,inner sep=2pt}]
\matrix (m) [matrix of math nodes, row sep=1em,
column sep=1.8em, text height=1.5ex, text depth=0.25ex]
{%&\GG_{m+1}&&\g_{m+1}\\~\\
&I_{\GG_{m+1}}(\sigma)\\
\msf{C}_{m+1}(Y)&\msf{P}_{m+1}(Y)\cdot \Gamma^1\\
&\msf{P}_{m+1}(Y)\cdot J_\theta&\Gamma^1&\ggama_1&\mfr{f}\\
\msf{P}_{m+1}(Y)&&J_\theta&\msf{J}_\theta&\mfr{j}_\theta\\
&&\Gamma^2&\ggama_2&\set{0}\\
&\msf{C}_{m+1}(Y)\cap \Gamma^1&\\
&\msf{C}_{m+1}(Y)\cap \Gamma^2\\~\\
&\set{1}&&&\\}; 

\path[-,font=\scriptsize]

%(m-1-2) edge[thin] node[auto] {} (m-3-2)
%(m-1-4) edge[thin] node[auto] {} (m-5-4)
(m-1-2) edge[thin] node[auto] {} (m-2-2)
		edge[thin, bend left=10] node[auto] {} (m-3-3)
(m-2-2) edge[thin] node[auto] {} (m-3-3)
		edge[thin] node[auto] {} (m-3-2)
(m-3-2) edge[thin] node[auto] {} (m-4-3)
(m-3-3) edge[thin] node[auto] {} (m-4-3)
(m-4-1)	edge[bend right=1,thin] node[auto] {} (m-6-2)
		edge[bend right=15,thin] node[auto] {} (m-7-2)
		edge[thin] node[auto] {} (m-3-2)
		edge[thin, bend left=5] node[auto] {} (m-2-2)
		edge[thin] node[auto] {} (m-2-1)
(m-2-1) edge[thin] node[auto] {} (m-1-2)
(m-5-3) edge[thin] node[auto] {} (m-4-3)
(m-7-2) edge[thin, bend right=5] node[auto] {} (m-5-3)
		edge[thin] node[auto] {} (m-9-2)
(m-6-2) edge[thin] node[auto] {} (m-4-3)
		edge[thin] node[auto] {} (m-7-2)			
(m-3-4) edge[thin] node[auto] {} (m-4-4)
(m-5-4) edge[thin] node[auto] {} (m-4-4)	

(m-3-5) edge[thin] node[auto] {} (m-2-5)
		edge[thin] node[auto] {} (m-4-5)
(m-5-5) edge[thin] node[auto] {} (m-4-5)		
;

\path[->,font=\scriptsize]
(m-3-3) edge[dashed] node[below] {$\llog$} node[above] {1-1} (m-3-4)
(m-4-3)	edge[dashed] node[below] {$\llog$} node[above] {1-1} (m-4-4)
(m-5-3) edge[thin] node[below] {$\llog$} node[above] {$\sim$} (m-5-4)	
;
\path[->>,font=\scriptsize]
(m-3-4) edge[thin] node[auto] {} (m-3-5)
(m-4-4) edge[thin] node[auto] {} (m-4-5)
(m-5-4) edge[thin] node[auto] {} (m-5-5)
;
\end{tikzpicture}\caption{Diagram for even level, not divisible by $4$, $\ell=2$}\label{figure:4k+1}\end{figure}

Let $\vartheta'\in\irr(J_\theta)$ be some extension of $\vartheta$ to $J_\theta$. As the group $\msf{P}_{m+1}(Y)$ is abelian and normalizes $J_\theta$, we can extend $\vartheta'$ to a linear character $\hat{\vartheta}$ of $\msf{P}_{m+1}(Y)\cdot J_\theta$ and which lies under $\varphi$. Repeating the proof of Proposition~\ref{propo:Theta}(1) verbatim, one shows that the induced character $\Theta:=\ind_{\msf{P}_{m+1}(Y) J_\theta}^{\msf{P}_{m+1}\Gamma^1}(\hat{\vartheta})$ is irreducible, of degree $q=\abs{\mfr{j}_\theta}$ and lies under $\varphi$. Additionally, by the bijection in Proposition~\ref{propo:lift}(2) and by the proof of Proposition~\ref{propo:Theta}(2), one obtains that the inertia subgroup of $\Theta$ in $\GG_{m+1}$ is $I_{\GG_{m+1}}(\sigma)$.

Finally, by Lemma~\ref{lem:conjclassGtoG1}, we know that \[\msf{C}_{m+1}\Gamma^1/\msf{P}_{m+1}(Y)\Gamma^1= \St_{G}^{-r}(Y)/\St_{G^1}^{-r}(Y)\cong\SL_1(\dbF_{q^2}\mid\dbF_q)\]
is a finite cyclic group. Hence, by \cite[Corollary~11.22]{Isaacs}, the character $\Theta$ extends to a character $\hat{\Theta}$ of $I_{\GG_{m+1}}(\sigma)$. Inducing further to $\GG_{m+1}$, we have that the pullback of $\ind_{I_{\GG_{m+1}}(\sigma)}^{\GG_{m+1}}(\hat{\Theta})$ to $G$ is equal to $\varphi$ and 
\[\varphi(1)=q\cdot\abs{G_{m+1}:I_{\GG_{m+1}}(Y)}=q^{\frac{m}{2}}=d_{m}^2(q).\]

Theorem~\ref{theo:proveseverything} is now proved.\qed

\section{The Representation Zeta Function of $\SL_1(D)$}\label{section:zeta}

At last, we are ready to complete the proof of Theorems~\ref{theo:zeta} and \ref{theo:leveldimension}.

Note that Theorem~\ref{theo:zeta} in fact follows from Theorem~\ref{theo:leveldimension}. Indeed, as can be verified by a formal comparison of the formula for representation zeta function suggested in Theorem~\ref{theo:zeta} and the generating function $\sum_{m=0}^\infty a_m^\ell(q)d^\ell_m(q)^{-s}$, we have that
\[\zeta_{\SL_1(D)}(s)=\frac{\frac{q^\ell-1}{q-1}(1-q^{-\binom{\ell}{2}s})+\left(\frac{q^\ell-1}{\iota_\ell(q)\cdot(q-1)}\right)^{-s}\cdot\iota_\ell(q)^2\cdot (q-1)\cdot\left(\sum_{\lambda=0}^{\ell-2}q^{\lambda(1-\frac{\ell-1}{2}\cdot s)}\right)}{1-q^{(\ell-1)-\binom{\ell}{2}s}}.\]
at any $s\in\dbC$ with $\Re(s)>\frac{2}{\ell}.$ Here $a_m^\ell(t),\:d_m^\ell(t)$ are the integer valued functions given in \eqref{equa:adefin} and \eqref{equa:ddefin}. This implies the equality of the two functions.

Also, recall that the fact that every irreducible character of $G$ of level $m$ is of dimension $d_m^\ell(q)$ has already been proved (Theorem~\ref{theo:proveseverything}). Thus, we are only left to prove that the set $\irr^m(G)$, of irreducible characters of level $m$, consists of precisely $a_m^\ell(q)$ elements, for all $m\in\dbN_0$. This follows from a simple computation, which we outline here.

Note that for any $m\in\dbN_0$ the union $\bigsqcup_{k=0}^{m}\irr^k(G)$ can be naturally identified with the set $\irr\left(\GG_{m+1}\right).$ 
As $\GG_1\cong{\SL_1(\dbF_{q^\ell}\mid\dbF_q)}$ is an abelian group of order $\frac{q^\ell-1}{q-1}$, it is clear that $\abs{\irr^0(G)}=\frac{q^\ell-1}{q-1}=a_0^\ell(q)$.

For characters of positive level, note that for any $r\in\dbN$ the reduced norm induces a surjection $\NNrd_r:\OO^\times/(1+\PP^r)\project\oo^\times/(1+\pp^{\ceil{\frac{r}{\ell}}})$ and that the natural isomorphism $\GG_r=G/G^r\cong G({1+\PP}^r)/({1+\PP}^r)$ identifies $\GG_r$ with $\ker(\NNrd_r)$ (see e.g. \cite[\S~3]{riehm}). Thus, \[\abs{\GG_r}=\frac{\abs{{\OO^\times}:({1+\PP}^r)}}{\abs{\oo^\times:(1+\pp^{\ceil{r/\ell}})}}=\frac{q^{\ell}-1}{q-1}\cdot q^{(r-1)\ell-\ceil{\frac{r}{\ell}}+1},\quad\text{for all } r\ge 1.\]
Moreover, for any $m\in\dbN$ we have that
\begin{align*}
\frac{q^\ell-1}{q-1}q^{m\ell-\ceil{\frac{m+1}{\ell}}+1}&=\abs{\GG_{m+1}}=\zeta_{\GG_{m+1}}(-2)&\\
&=\sum_{k=0}^m\sum_{\rho\in \irr^{k}(G)}\dim(\rho)^2\\%&=\sum_{\rho\in \irr^{\le m-1}(G)}\dim(\rho)^2+\sum_{\rho\in\irr^m(G)}\dim(\rho)^2\\
&=\zeta_{\GG_{m}}(-2)+\abs{\irr^m(G)}\cdot d_m^\ell(q)^2&\text{ (by Theorem~\ref{theo:proveseverything})}\\
&=\frac{q^{\ell}-1}{q-1}q^{(m-1)\ell-\ceil{\frac{m}{\ell}}+1}+\abs{\irr^m(G)}\cdot d^\ell_m(q)^2\end{align*}
The fact that $\abs{\irr^m(G)}=a_m^\ell(q)$ now follows by direct computation, noticing that $\ceil{\frac{m+1}{\ell}}=\ceil{\frac{m}{\ell}}$ whenever ${\ell\nmid m}$.
\qed

\bibliographystyle{siam}
\bibliography{Characters}
\end{document}